\documentclass[pdflatex,sn-mathphys-num]{sn-jnl}

\usepackage{graphicx}%
\usepackage{multirow}%
\usepackage{amsmath,amssymb,amsfonts}%
\usepackage{amsthm}%
\usepackage{mathrsfs}%
\usepackage[title]{appendix}%
\usepackage{xcolor}%
\usepackage{textcomp}%
\usepackage{manyfoot}%
\usepackage{booktabs}%
\usepackage{enumitem}
\usepackage{verbatim}

\theoremstyle{thmstyleone}%
\newtheorem{theorem}{Theorem}[section]%
\newtheorem{proposition}[theorem]{Proposition}%
\newtheorem{lemma}[theorem]{Lemma}%
\newtheorem{remark}[theorem]{Remark}%
\newtheorem{assumption}[theorem]{Assumption}%

\newcommand{\<}{\langle}
\renewcommand{\>}{\rangle}
\newcommand{\grad}{\nabla}
\newcommand{\hess}{\nabla^2}
\newcommand{\sym}{\mathrm{Sym}}
\newcommand{\NN}{\mathbb{N}}
\newcommand{\ZZ}{\mathbb{Z}}
\newcommand{\RR}{\mathbb{R}}
\renewcommand{\SS}{\mathbb{S}}
\newcommand{\posint}{\ZZ_+}
\newcommand{\PP}{\mathbb{P}}
\newcommand{\EE}{\mathbb{E}}

\newcommand{\sB}{\mathcal{B}}

\newcommand{\sF}{\mathcal{F}}

\newcommand{\sL}{\mathcal{L}}

\newcommand{\sU}{\mathcal{U}}
\newcommand{\diff}{{\mathrm d}}
\newcommand{\ds}{\diff s}

\newcommand{\dx}{\diff x}
\newcommand{\dt}{\diff t}
\newcommand{\dr}{\diff r}

\newcommand{\dbs}{\diff B_{s}}

\newcommand{\ind}{\mathbf{1}}
\newcommand{\id}{\mathbf{I}}
\newcommand{\lipone}{\mathrm{Lip}}
\newcommand{\ig}{\text{L}}
\newcommand{\op}{\mathrm{op}}
\newcommand{\smoothlip}{\mathrm{Lip}^{\infty}_b}

\newcommand{\indicator}{\mathbf{1}}
\newcommand{\transpose}{\mathrm{T}}
\newcommand{\hs}{{}}
\newcommand{\trace}{\text{Tr}}
\newcommand{\wass}{d}
\newcommand{\sphere}{\SS^{N-1}}
\newcommand{\magn}{S_n}

\begin{document}

\title[Article Title]{The rate of convergence of the critical mean-field $O(N)$ magnetization via multivariate nonnormal Stein's method}

\author[1]{\fnm{Timothy M.} \sur{Garoni}}\email{tim.garoni@monash.edu}
\author[1]{\fnm{Aram} \sur{Perez}}\email{aram.perez@monash.edu}
\author[1]{\fnm{Zongzheng} \sur{Zhou}}\email{eric.zhou@monash.edu}
\affil[1]{\orgdiv{School of Mathematics}, \orgname{Monash University}, \orgaddress{\city{Clayton}, \postcode{3800}, \state{VIC},
    \country{Australia}}} 

\abstract{
  We study the distribution of the magnetization of the critical mean-field $O(N)$ model with $N\ge2$. Specifically, we 
  bound the Wasserstein distance between the finite-volume and limiting distributions, in terms of the number of spins. To achieve this, we extend a recent
  multivariate nonnormal approximation theorem. This generalizes known results for the Curie-Weiss magnetization to the
  multivariate $O(N)$ setting.
}


\maketitle

\section{Introduction}
The $O(N)$ model was introduced in~\cite{Stanley1968} as a natural generalization of the Ising, $XY$ and Heisenberg
models. For integer $N,n\ge1$, let $\sphere:=\{x\in\RR^N:|x|=1\}$ denote the set of Euclidean unit vectors in $\RR^N$, and let $P_{N,n}$
denote the $n$-fold product of uniform measure on $\sphere$. The mean-field $O(N)$ model is then defined for $\beta\ge0$ by the probability
measure $\PP_{N,n,\beta}$ whose density with respect to $P_{N,n}$ is proportional to
\begin{equation}
  \exp\left(\frac{\beta}{2n}\sum_{i,j=1}^n \sigma_i\cdot\sigma_j\right), \qquad \sigma\in (\sphere)^n.
  \label{eq:mean-field O(N) measure}
\end{equation}

It is known that, with respect to the mean-field measure~\eqref{eq:mean-field O(N) measure}, the magnetization
\begin{equation}
  \label{eq:magnetization definition}
\magn(\sigma):=\sum_{i=1}^n\sigma_i
\end{equation}
obeys a law of large numbers iff $\beta\le N$; see~\cite{EllisNewman1978,Ellis1985} for the $N=1$ case, \cite{KirkpatrickMeckes2013} for
the $N=3$ case, and~\cite{KirkpatrickNawaz2016} for the general case of $N\ge2$.
In the subcritical phase, $\beta<N$, it is also known~\cite{EllisNewman1978,Ellis1985,KirkpatrickMeckes2013,KirkpatrickNawaz2016} that
$S_n/\sqrt{n}$ obeys a Gaussian central limit theorem. By contrast, at the critical point, $\beta=N$,
it has long been known that $n^{-3/4}\,S_n$ converges weakly to a random vector in $\RR^N$ with Lebesgue density
proportional to $\exp(-a_N |x|^4)$, where $a_N:=N^2/(4N+8)$. This was established for $N=1$ in~\cite{SimonGriffiths1973},
and then generalised to $N\ge2$ in~\cite{DunlopNewman1975}.
Analogous limit theorems have been subsequently established for other mean-field spin models, including the Curie-Weiss-Potts
model~\cite{EllisWang1990}, and Blume-Emery-Griffiths model~\cite{EllisOttoTouchette2005}.

A natural question associated with a given limit theorem is to determine its rate of convergence. It was shown
in~\cite{EichelsbacherLowe2010} via Stein's method, that for the $N=1$ model with $\beta<1$, the Kolmogorov distance
between the distribution of $\sqrt{1-\beta}\, S_n/\sqrt{n}$ and a standard univariate normal is bounded above by
$c(\beta)/\sqrt{n}$. Utilising a multivariate generalisation of Stein's method established
in~\cite{ChatterjeeMeckes2008,ReinertRollin2009,Meckes2009}, analogous results~\cite{KirkpatrickMeckes2013,KirkpatrickNawaz2016} were then
shown to hold in the disordered phase, $\beta<N$, of the $O(N)$ model for any $N\ge2$.

In the critical case, it was shown in~\cite{EichelsbacherLowe2010,ChatterjeeShao2011} using nonnormal extensions of Stein's method,
that for the $N=1$ model with $\beta=1$ the Kolmogorov distance between the distribution of $n^{-3/4}\, S_n$ and the probability
measure with Lebesgue density proportional to $\exp(-x^4/12)$ is again bounded above by $c/\sqrt{n}$. We establish here a natural
multivariate generalisation of this result to all $N\ge 2$.

Our method relies on a general multivariate nonnormal Stein's method framework established
in~\cite{FangShaoXu2019,FangShaoXu2019correction}, which in turn can be viewed as a multivariate generalization
of~\cite{ChatterjeeShao2011}. 
As explicitly noted in~\cite{FangShaoXu2019}, however, their underlying assumptions
do not hold for densities of the form $\exp(-a|x|^m)$ with $m>2$, and so their results cannot be applied to the critical mean-field $O(N)$
model. Nevertheless, as shown here, a generalization of~\cite[Theorem 2.5]{FangShaoXu2019} can be established which does indeed cover the
case of the critical $O(N)$ model. We present this result in Theorem~\ref{Theorem: main result} of Section~\ref{sec: Stein's method}.

The distributional bounds we present are in terms of the Wasserstein distance. The Wasserstein distance between the laws $\sL(X)$, $\sL(Y)$
of two random vectors $X,Y\in\RR^d$ is 
\begin{equation}
  \wass(\sL(X),\sL(Y)):= \inf_{(X,Y)}\EE|X-Y| = \sup_{h \in \lipone}\left|\EE\, h(X) - \EE\, h(Y)\right|
  \label{eq: Wasserstein definition}
\end{equation}
where the infimum is taken over all couplings of $\sL(X)$ and $\sL(Y)$, and $\lipone$ denotes the set of all functions $h:\RR^d\to\RR$ such
that $|h(x)-h(y)|\le |x-y|$ for all $x,y\in\RR^d$. We note that convergence in Wasserstein distance implies weak convergence; see
e.g.~\cite{GibbsSu2002}.  

In the disordered phase, the bounds originally presented in~\cite{KirkpatrickMeckes2013,KirkpatrickNawaz2016} were not in terms of the
Wasserstein distance, but instead in terms of an integral probability metric defined by a smaller, smoother, class of test functions. These
bounds were proved by applying a general multivariate normal approximation theorem presented in~\cite{Meckes2009}. However, recent results
giving multivariate normal approximation in the Wasserstein distance allow the results from~\cite{KirkpatrickNawaz2016} to be immediately
sharpened to the Wasserstein distance. Indeed, substituting results given in~\cite[Lemmas 1 and 2]{KirkpatrickNawaz2016} into~\cite[Theorem
  1.1]{FangKoike2022} yields the following. 
\begin{theorem}[\cite{KirkpatrickMeckes2013,KirkpatrickNawaz2016,FangKoike2022}]
  \label{Theorem: high temp O(N)}
  Fix $N \geq 2$ and let $W_{n} := \sqrt{N-\beta}\,S_{n}/\sqrt{n}$ and $Z \sim N\left(0, \id\right)$, where $\id\in\RR^{d\times d}$ is the
  identity matrix. For any $\beta < N$ there exists $c(N,\beta)<\infty$ such that
    \begin{equation}
         \wass(\sL(W_n),\sL(Z)) \leq \frac{c(N,\beta)}{\sqrt{n}}.
    \end{equation}
\end{theorem}

Our main result in the current work is the analogous multivariate nonnormal approximation holding at criticality.
\begin{theorem}
  \label{Theorem: Critical O(N)}
  Fix $N \geq 2$ and let $W_{n} := n^{-3/4}\,S_{n}$. If $\beta=N$, then there exists $c(N)<\infty$ such that
    \begin{equation}
        \wass(\sL(W_n),\mu) \leq \frac{c(N)}{\sqrt{n}},
    \end{equation}
  where $\mu$ has Lebesgue density proportional to $\exp\left(-a_N |x|^{4}\right)$ and $a_N:=N^2/(4N+8)$.
\end{theorem}

\begin{remark}
As a corollary of Theorem~\ref{Theorem: Critical O(N)} it follows that for any $N\ge2$, the Wasserstein distance between the distribution of
$n^{-3/4}\,|S_n|$ and the probability measure with Lebesgue density proportional to $r^{N-1}\exp(-a_N r^4)$, is again bounded above by
$c(N)/\sqrt{n}$. This is analogous to the univariate nonnormal approximation presented in~\cite[Theorem 6]{KirkpatrickNawaz2016} for
$|S_n|^2$, however the bound given there is not in terms of Wasserstein distance, but in terms of an integral probability metric defined by
a smaller class of test functions, and contains an additional logarithmic factor.
\end{remark}

The remainder of this paper is organised as follows. Section~\ref{sec: Stein's method} introduces the relevant background on Stein's method, and
provides the statement of our main result on Wasserstein approximation, Theorem~\ref{Theorem: main result}. In Section \ref{sec: O(N) model}
we apply Theorem~\ref{Theorem: main result} to prove Theorem~\ref{Theorem: Critical O(N)}. 
Section~\ref{sec: sdes and semigroups} provides preliminary results related to a class of stochastic differential equations, and their
related stochastic semigroups, which provide the framework for the proof of Theorem \ref{Theorem: main result} given in Section \ref{sec: proof of main
  result}. The key result used in this proof is a bound on the derivatives of the relevant stochastic semigroups.
These bounds are proved in Section~\ref{sec: semigroup bounds}, using Elworthy-Li formulae and
bounds on the spatial derivatives of the solutions to the corresponding stochastic differential equation
proved in Section~\ref{sec: variation process bounds}.
Finally, the appendix contains the proof of a proposition stated in Section~\ref{sec: sdes and semigroups}.

\subsection{Notation}
\label{subsec:notation}
We will denote the set of positive integers by $\posint$, and the set of natural numbers by $\NN:=\posint\cup\{0\}$.
For $d\in\posint$ and $x,y\in \RR^{d}$, we let $\<x,y\>$ denote the Euclidean inner product, and set $|x| = \sqrt{\<x,x\>}$. For two
matrices $A, B \in \RR^{d \times d}$ we let $\<A,B\>_{\hs} = \sum_{i, j = 1}^{d}A_{i, j}B_{i, j}$ denote the Hilbert-Schmidt inner product,
and denote the corresponding norm by $\|A\|_{\hs} = \sqrt{\< A, A \>_{\hs}} = \sqrt{\trace(A^{\transpose}A)}$. We will denote the operator
norm of $A\in\RR^{d\times d}$ by $\|A\|_{\op}:=\sup_{x\in\SS^{d-1}}|Ax|$.

For open sets $U_1\subseteq \RR^\ell$ and $U_2\subseteq\RR^m$ we let $C^{k}(U_1;U_2)$ denote the set of $k$-times continuously differentiable
$f:U_1\to U_2$. If $U_2=\RR$ we abbreviate $C^k(U_1):=C^k(U_1,U_2)$, and if additionally $U_1=\RR^d$ we write simply $C^k$.
We let $B_b$ denote the
space of $f:\RR^d\to\RR$ which are bounded and Borel measurable, and let $C_b$ denote the subspace of $B_b$ of uniformly continuous such functions. We will then make
use of the convenient abbreviation $\smoothlip:=\lipone\cap C^{\infty}\cap C_b$.

For $f\in C^1(\RR^d;\RR)$, we let $D_u f(x)$ denote the directional derivative of $f$ in the direction $u\in\RR^d$, and we extend the
definition to $f\in C^1(\RR^d;\RR^m)$, by then defining $D_u f(x)$ entrywise. For given $f\in C^3(\RR^d;\RR^m)$ and $u,v,w,x\in\RR^d$ we
define
\begin{equation}
  \label{eq: D notation}
  \begin{split}
    Df(x)[u]&:= D_u f(x)\\
    D^2f(x)[u,v]&:=D_v D_u f(x)\\
    D^3f(x)[u,v,w]&:=D_w D_v D_u f(x)
  \end{split}
\end{equation}
We note that for each given $x\in\RR^d$ and $i=1,2,3$, the map $D^i g(x) : (\RR^d)^i\to\RR^d$ is multilinear and symmetric.

For $f \in C^{1}(\RR^d;\RR)$, we let $\grad f$ denote the gradient of $f$, so that
\begin{equation}
  Df(x)[u]=\<\grad f(x),u\>,
\end{equation}
and set $\|\grad f\|_\infty:=\sup_{x \in \RR^{d}}|\grad f(x)|$. 
Finally, for $f \in C^{2}(\RR^{d};\RR)$, we let $\hess f$ denote the Hessian of $f$,
so that
\begin{equation}
D^2 f(x)[u,v] =\<u,\hess f(x) v\>,
\end{equation}
and let $\Delta$ denote the Laplacian.

\section{Stein's Method}
\label{sec: Stein's method}
Our goal in this section is to describe bounds in Wasserstein distance between the distribution of an $\RR^d$-valued random vector and 
the probability measure with Lebesgue density proportional to $e^{-V}$, for sufficiently well-behaved $V$. Specifically,
we focus on $V$ which satisfy the following.

\begin{assumption}\label{assumption: assumptions on V}
  The function $V$ belongs to $C^{4}(\RR^{d}, \RR)$, and there exist constants $M_{1},M_{2}\in(0,\infty)$ and a continuous
  function $\rho : \RR^{d} \to \RR$ satisfying
  \begin{equation}
    \label{eq: rho bounds}
    c_1\,|x|^{k}\,\ind(|x|\geq B) \leq \rho(x) \leq c_2\,|x|^{k}, \qquad \forall\, x\in\RR^d,
  \end{equation}
  for constants $k\in[0,\infty)$, $B\in[0,\infty)$ and $c_1, c_2\in(0,\infty)$, such that for all $x\in\RR^d$ and all nonzero $u,v,w\in\RR^d$:
    \begin{align}\label{eq: Assumption 1}
        \< u, \hess V(x)  u\> &\geq \rho(x)|u|^{2},\\
        |D_{v} D_{u} \grad V(x)| &\leq M_{1} \frac{1 + \rho(x)}{1+|x|}|u||v|, \label{eq: Assumption 2}\\
        |D_{w} D_{v} D_{u} \grad V(x)| &\leq M_{2} \frac{1 + \rho(x)}{(1+|x|)^{2}}|u||v||w|. \label{eq: Assumption 3}
    \end{align}
    If $k=0$, there additionally exists $M_0\in(0,\infty)$ such that
    \begin{equation}
      \label{eq: Assumption 0}
      |D_{u} \grad V(x)| \leq M_{0}|u|.
    \end{equation}
\end{assumption}

We will prove the following general theorem in Wasserstein approximation.
\begin{theorem}\label{Theorem: main result}
  Suppose $V$ satisfies Assumption \ref{assumption: assumptions on V}, and let $\mu$ have Lebesgue density proportional to
  $e^{-V}$.
  Let $W,W'$ be identically distributed $\RR^d$-valued random variables, defined on the same probability space, such that if $\delta:=W'-W$ then
  $W$ and $\delta\delta^\transpose$ are both integrable. 
  Let $\sF\supseteq\sigma(W)$ be a $\sigma$-algebra, let $\lambda>0$, and define $R_1$ and $R_2$ so that
    \begin{equation}
      \mathbb{E}[\delta|\mathcal{F}] = \lambda(-\grad V(W) + R_{1}),
      \label{eq: Regression Property}
    \end{equation}
    and
    \begin{equation}
      \mathbb{E}\left[\delta\,\delta^{\transpose}\big|\mathcal{F}\right] = 2\lambda(\id + R_{2}).
      \label{eq: Correlation Property}
    \end{equation}
    Then there exists a constant $C\in(0,\infty)$ such that
    \begin{equation}
      d(\sL(W),\mu)
      \leq C\left\lbrace \frac{1}{\lambda}\mathbb{E} [|\delta|^{3}(|\log|\delta|| \vee 1)] + \mathbb{E}|R_{1}|
      + \sqrt{d}\,\mathbb{E}\,\|R_{2}\| \right\rbrace. 
    \end{equation}
\end{theorem}

The first step in proving Theorem~\ref{Theorem: main result} is to consider the Stein equations corresponding to $V$.
For differentiable $V:\RR^d\to\RR$, let $\mu$ be the Lebesgue density proportional to $e^{-V}$, and
consider the differential operator
\begin{equation}
  \ig := \Delta - \< \grad V, \grad \>.
\end{equation}
For $h\in\lipone$, the corresponding Stein equation is 
\begin{equation}\label{eq: Stein equation}
    \ig f = h - \mu(h),
\end{equation}
where $\mu(h)$ denotes the expectation of $h$ with respect to $\mu$.

\begin{assumption}\label{Assumption: bounds on stein solution}
  The function $V$ belongs to $C^1(\RR^d,\RR)$, and is such that $e^{-V(x)}$ and $|x|e^{-V(x)}$ are Lebesgue integrable, and the following
  hold. There exists $C\in(0,\infty)$ such that for each $h\in\smoothlip$, equation~\eqref{eq: Stein equation} has a solution $f = f_{h}
  \in C^{2}(\RR^{d}, \RR)$ satisfying for all nonzero $u,v\in\RR^d$ and $\epsilon>0$:
    \begin{equation}\label{eq: bound gradient f}
        \|\grad f\|_\infty \leq C,
    \end{equation}
    \begin{equation}\label{eq: Bound hilbert schmidt norm}
        \sup_{x \in \mathbb{R}^{d}}| D_{v} D_{u}f(x)| \leq C\, |u|\,|v|,
    \end{equation}
    \begin{equation}\label{eq: Bound hilbert schmidt norm 2}
        \sup_{x, y \in \mathbb{R}^{d}: |y|\leq 1}| D_{v} D_{u}f(x + \epsilon y) - D_{v} D_{u}f(x)| \leq C
        |\epsilon|(|\log \epsilon | \vee 1) \,|u|\,|v|.
    \end{equation}
\end{assumption}

Our main strategy in proving Theorem~\ref{Theorem: main result} follows~\cite{FangShaoXu2019}, and amounts to first establishing Wasserstein
approximation under Assumption~\ref{Assumption: bounds on stein solution}, and then showing that any $V$ satisfying
Assumption~\ref{assumption: assumptions on V} also satisfies Assumption~\ref{Assumption: bounds on stein solution}.

To that end, we state the following variant of~\cite[Theorem 2.5]{FangShaoXu2019} which establishes Wasserstein approximation under
Assumption~\ref{Assumption: bounds on stein solution}. For completeness, we provide a proof here, since for our application to the $O(N)$
model it is convenient to slightly generalise the statement from~\cite{FangShaoXu2019} to allow arbitrary $\sigma$-algebras
$\mathcal{F}\supseteq\sigma(W)$ in the definition of $R_1$ and $R_2$. 
Moreover, adding a suitable truncation and smoothing argument to the proof implies that
in Assumption~\ref{Assumption: bounds on stein solution},
equations~\eqref{eq: bound gradient f},~\eqref{eq: Bound hilbert schmidt norm} and~\eqref{eq: Bound hilbert schmidt norm 2}
need only be assumed to hold for $h\in\smoothlip$, rather than all $h\in\lipone$. 

\begin{theorem}[\cite{FangShaoXu2019}]\label{Theorem: Wasserstein Bound}
  Suppose $V$ satisfies Assumption \ref{Assumption: bounds on stein solution}, and let $\mu$ have Lebesgue density proportional to $e^{-V}$.
  Let $W,W'$ be identically distributed $\RR^d$-valued random variables, defined on the same probability space, such that if $\delta:=W'-W$ then
  $W$ and $\delta\delta^\transpose$ are both integrable.
  Let $\sF\supseteq\sigma(W)$ be a $\sigma$-algebra, let $\lambda>0$, and define $R_1$ and $R_2$ via~\eqref{eq: Regression Property} and~\eqref{eq: Correlation Property}.
  Then there exists a constant $C\in(0,\infty)$ such that
    \begin{equation}
      d(\sL(W),\mu)
      \leq C\left\lbrace \frac{1}{\lambda}\mathbb{E} [|\delta|^{3}(|\log|\delta|| \vee 1)] + \mathbb{E}|R_{1}|
      + \sqrt{d}\,\mathbb{E}\,\|R_{2}\| \right\rbrace. 
    \end{equation}
\end{theorem}

\begin{proof}
    Fix $h \in \lipone$, and for each $n\in\posint$ define $h_{n}(x) = \min\{n, \max\{h(x), -n\}\}$ and $\tilde{h}_{n}(x) = \EE[h_n(x + Z/n)]$, where $Z \sim
    \mathcal{N}(0, \id)$. It is straightforward to verify that $\tilde{h}_{n}\in\smoothlip$.
    
    Let $f = f_{\tilde{h}_n}$ be the corresponding solution to equation \eqref{eq: Stein equation} guaranteed by Assumption~\ref{Assumption:
      bounds on stein solution}. By Taylor expanding $f(W')$ around $W$,
    \begin{equation}\label{eq: exchangeability}
        \begin{split}
     0&= \EE[f(W')]-\EE[f(W)]\\
      &= \EE[\< \grad f(W), \delta\>] + \frac{1}{2}\EE[\< \hess f(W), \delta \delta^{\transpose} \>_{\hs}]\\
      &\quad + \int_{0}^{1}(1-t)\,\EE[\< \hess f(W + t \delta) - \hess f(W), \delta \delta^{\transpose} \>_{\hs}]\, \dt.
        \end{split}
    \end{equation}
    Using the definition of $R_1$ and the fact that $\sigma(W) \subseteq \mathcal{F}$, we have
    \begin{equation}
      \label{eq: R1 identity}      
        \begin{split}
            \EE[\< \grad f(W), \delta\>] &= \EE[\< \grad f(W), \EE[\delta|\mathcal{F}]\>] = -\lambda\,\EE[\<
              \grad f(W), \grad V(W)\>] + \lambda\,\EE[\< \grad f(W),R_{1}\>]. 
        \end{split}
    \end{equation}
    Similarly, using the definition of $R_2$ we have
    \begin{equation}
      \label{eq: R2 identity}
        \begin{split}
            \frac{1}{2}\EE[\< \hess f(W), \delta \delta^{\transpose} \>_{\hs}] = \lambda \EE[\Delta f(W)] + \lambda
            \EE[\< \hess f(W), R_{2}  \>_{\hs}]. 
        \end{split}
    \end{equation}
    Substituting~\eqref{eq: R1 identity} and~\eqref{eq: R2 identity} into \eqref{eq: exchangeability} and rearranging yields
    \begin{equation}\label{eq: first expression ig W}
    \begin{split}
        \EE[\ig f(W)] &= - \EE[\< \grad f(W),R_{1}\>] - \EE[\< \hess f(W), R_{2}  \>_{\hs}]\\
        & \quad - \frac{1}{\lambda}\int_{0}^{1}(1-t)\, \EE[\< \hess f(W + t \delta) - \hess f(W), \delta
          \delta^{\transpose} \>_{\hs}]\, \dt. 
    \end{split}
    \end{equation}
    By \eqref{eq: bound gradient f} we have
    \begin{equation}
      \label{eq: Stein first term}
        \left|\EE[\< \grad f(W),R_{1}\>]\right| \leq \|\grad f\|_\infty\,\EE|R_{1}|\leq C\,\EE|R_{1}|.
    \end{equation}
    And using $\frac{1}{\sqrt{d}}\|A\|_{\hs} \leq \|A\|_{\text{op}} = \sup_{|v|, |w| = 1}|\< A,
    vw^{\transpose}\>_{\hs}|$ and \eqref{eq: Bound hilbert schmidt norm} yields
    \begin{equation}
      \label{eq: Stein second term}
      \left|\EE[\< \hess f(W), R_{2}  \>_{\hs}]\right| \leq \sup_{x \in \RR^{d}}\|\hess f(x)\|_{\hs} \,\EE\|R_{2}\|
      \leq C \,\sqrt{d}\,\EE\|R_{2}\|. 
    \end{equation}
    Finally, using~\eqref{eq: Bound hilbert schmidt norm 2} we obtain
    \begin{equation}
      \begin{split}
        \label{eq: Stein third term}
      &\left|\int_{0}^{1}(1-t)\,\EE\< \hess f(W + t \delta)-\hess f(W), \delta \delta^{\transpose} \>_{\hs}\, \dt\right|\\ 
      &\quad = \left|\int_{0}^{1}(1-t)\,
      \EE\left[D_\delta D_\delta f\left(W+ |\delta|\, \frac{t \delta}{|\delta|}\right) - D_\delta D_\delta f(W)\right]
      \, \dt\right|,
      \\ 
      & \quad \leq \frac{C}{2}\,\mathbb{E} [|\delta|^{3}(|\log|\delta|| \vee 1)].
    \end{split}
    \end{equation}
    Applying the bounds~\eqref{eq: Stein first term}, \eqref{eq: Stein second term} and~\eqref{eq: Stein third term} to~\eqref{eq: first
      expression ig W} then implies that for all $n\in\posint$
    \begin{equation}\label{eq: prelim bound wass}
        \begin{split}
            \left|\EE\left[\tilde{h}_{n}(W) - \mu\left(\tilde{h}_{n}\right)\right] \right| &= \left|\EE[\ig f_{\tilde{h}_n}(W)]\right|\\
            & \leq C\,\left\lbrace \frac{1}{\lambda}\mathbb{E} [|\delta|^{3}(|\log|\delta|| \vee 1)] + \mathbb{E}|R_{1}| +
            \sqrt{d}\,\mathbb{E} \|R_{2}\| \right\rbrace. 
        \end{split}
    \end{equation}
    
Now, it is straightforward to show that for all $x\in\RR^d$ we have
\begin{equation}
  \lim_{n\to\infty} \tilde{h}_n(x) = h(x).
\end{equation}
Moreover, since $h\in\lipone$ and $|h_n(x)|\le |h(x)|$ for all $x\in\RR^d$, it follows that
\begin{equation}
  |\tilde{h}_n(x)| \le |h(0)| + \EE|Z| + |x|.
\end{equation}
Therefore, since by assumption $W$ and $\mu$ have finite first absolute moments, it follows by dominated convergence that
\begin{equation}
  \lim_{n\to\infty} \EE\, \tilde{h}_n(W) = \EE \,h(W),
  \label{eq: E h(W) limit}
\end{equation}
and
\begin{equation}
  \lim_{n\to\infty} \mu(\tilde{h}_n) = \mu(h).
  \label{eq: mu(h) limit}
\end{equation}
Combining equations~\eqref{eq: E h(W) limit} and~\eqref{eq: mu(h) limit} then shows that
\begin{equation}
  \lim_{n\to\infty} \left|\EE\left[\tilde{h}_{n}(W) - \mu\left(\tilde{h}_{n}\right)\right] \right|
  =
 \left|\EE\left[h(W) - \mu\left(h\right)\right] \right|.
\label{eq: limit of discrepancy}
\end{equation}
The stated result then follows from~\eqref{eq: limit of discrepancy} by noting that the right-hand side of~\eqref{eq: prelim bound
  wass} is independent of $n$.
\end{proof}

 \section{\texorpdfstring{Rates of Convergence for the O(N) Model}{Rates of Convergence for the O(N) Model}}
 \label{sec: O(N) model}
In this section, we apply Theorem~\ref{Theorem: main result} to prove Theorem~\ref{Theorem: Critical O(N)}.
We therefore set $\beta=N$ throughout. We begin by stating a limit theorem for $S_n/n^{3/4}$, which is an immediate corollary
of~\cite[Theorem 1]{DunlopNewman1975}.
\begin{proposition}[\cite{DunlopNewman1975}]\label{prop: mgf convergence}
  Fix $N \geq 2$. Let $\sigma$ be distributed via~\eqref{eq:mean-field O(N) measure} with $\beta=N$, and let
  $W_{n} = n^{-3/4}\,S_{n}(\sigma)$. Let $Y$ be a random vector in $\RR^N$ whose distribution has Lebesgue density
  proportional to $\exp\left(-a_N |x|^{4}\right)$, with $a_N:=N^2/(4N+8)$. Then:
  \begin{enumerate}[label=(\roman*)]
    \item\label{prop-part: critical weak convergence} $W_n$ converges weakly to $Y$ as $n\to\infty$
    \item\label{prop-part: finite moments} $\EE|W_{n}|^{k} \leq C_{k} < \infty$, for all integer $k\ge0$ and $n\ge 2$
  \end{enumerate}
\end{proposition}

Before proceeding further, we introduce the following definitions which will prove useful. For $\nu>-1$ we define $U_\nu:\RR\to(0,\infty)$ via
\begin{equation}
  \label{eq: U definition}
  U_\nu(x) := \sum_{k=0}^\infty \frac{x^{2k}}{2^{2k}\,k!\,(\nu+1)^{(k)}}
\end{equation}
where $(\theta)^{(n)}:=\theta(\theta+1)\ldots(\theta+n-1)$ denotes the rising factorial. The function $U_\nu$ is readily expressed in terms of the
modified Bessel function of the first kind with order $\nu$, however we find working directly with $U_\nu$ more convenient in the current
context. We then further define the functions $f_{N}, g_{N}: \RR \to (0,\infty)$ via
\begin{equation}\label{eq: fN and gN}
  f_{N}(x) := \frac{U_{N/2}(N x)}{U_{N/2-1}(N x)},
  \qquad\qquad g_{N}(x) := \frac{N}{N+2} \frac{U_{N/2+1}(N x)}{U_{N/2-1}(N x)}.
\end{equation}
It can be shown that
\begin{equation}
  \label{eq: fN asymptotics}
f_N(x) = 1 - \frac{N}{N+2} x^2 + \frac{2N^2}{(N+2)(N+4)} x^4 + O(x^6), \qquad x\to 0
\end{equation}
and
\begin{equation}
  \label{eq: fN + gN identity}
f_N(x) + x^2 g_N(x) =1.
\end{equation}

Now consider a coupling $(\sigma,\sigma')$ of~\eqref{eq:mean-field O(N) measure} with itself, defined by randomly choosing $\sigma$
via~\eqref{eq:mean-field O(N) measure} and then generating $\sigma'$ from $\sigma$ by taking one step of the corresponding single-spin
update Glauber (heat-bath) chain. In more detail, $\sigma'$ is obtained from $\sigma$ by selecting a uniformly random vertex, independently
of $\sigma$, and then choosing $\sigma'$ by conditioning on the values of $\sigma_j$ for all $j$ other than the chosen vertex. We then set
$W=n^{-3/4} S_n(\sigma)$ and $W'=n^{-3/4} S_n(\sigma')$. We will also set $m:=S(\sigma)/n$, and for $1\le i \le n$ introduce 
\begin{equation}
m_{(i)}:=\frac{S(\sigma)-\sigma_i}{n}.
\end{equation}
We have the following lemma concerning $(W,W')$.

\begin{lemma}\label{lemma: Critical remainders}
    Let $(W, W')$ be as described above. Then
    \begin{equation}\label{eq: W' - W conditioned on sigma}
        \EE[W'-W|\sigma] = \frac{1}{Nn^{3/2}}\left(-\frac{N^{2}}{N+2}|W|^{2}W + R_{1}\right),
    \end{equation}
    where
    \begin{equation}\label{eq: remainder 1}
    \begin{split}
        R_{1} &= \frac{N}{n^{1/4}}\sum_{i=1}^{n}\left[m_{(i)}f_{N}(|m_{(i)}|) - mf_{N}(|m|)\right] \\
        & \quad + Nn^{3/4}\left(f_{N}(|m|) - 1 + \frac{N|m|^{2}}{(N+2)}\right)m.
    \end{split}
    \end{equation}
    and 
    \begin{equation}\label{eq: W' - W W' - W T conditioned on sigma}
    \begin{split}
        \EE[(W'-W)(W'-W)^{\transpose}|\sigma] &= 2\frac{1}{Nn^{3/2}}\left(\id + R_{2}\right),
    \end{split}
    \end{equation}
    where 
    \begin{equation}\label{eq: remainder 2}
        \begin{split}
            R_{2} &= \left[\frac{N}{2n}\sum_{i=1}^{n}\sigma_{i}\sigma_{i}^{\transpose} - \frac{1}{2}\id\right] -
            \frac{N}{\sqrt{n}}WW^{\transpose} + \frac{N}{n^{2}}\sum_{i=1}^{n}\sigma_{i}\sigma_{i}^{\transpose}\\ 
            &\quad + \frac{N}{2n}\sum_{i=1}^{n}\,g_{N}(|m_{(i)}|)\,|m_{(i)}|^2\,
            \left[\sigma_{i}m_{(i)}^{\transpose}+ m_{(i)} \sigma_{i}^{\transpose} - \frac{\id}{N}\right]
            + \frac{N}{2n}\sum_{i=1}^{n}g_{N}(|m_{(i)}|)m_{(i)}m_{(i)}^\transpose. 
        \end{split}
    \end{equation}
\end{lemma}
\begin{proof}
  We follow a similar strategy to the proof~\cite[Lemma 1]{KirkpatrickNawaz2016} (see also \cite[Lemma 11]{KirkpatrickMeckes2013}).
  The key ingredients are the small $x$ asymptotics for $f_N(x)$ given in~\eqref{eq: fN asymptotics}, combined with the fact
  that, as shown in~\cite[Theorem 3]{KirkpatrickNawaz2016}, the quantity $m$ concentrates at 0 as $n\to \infty$ when $\beta = N$.

   From the definition of $(\sigma,\sigma')$ we have
    \begin{equation}
      \begin{split}
      \EE[W'-W|\sigma] &= \frac{1}{n^{3/4}}\frac{1}{n}\sum_{i=1}^n \Big(\EE\left(\sigma_i|\{\sigma_j\}_{j\neq i}\right) -\sigma_i\Big)\\
      &= \frac{1}{n^{7/4}}\sum_{i=1}^{n} \EE(\sigma_{i}|\{\sigma_j\}_{j\neq i}) - \frac{1}{n} W
      \end{split}
    \end{equation}
    Using spherical coordinates~\cite{Blumenson1960} and an integral representation~\cite[3.71.(9)]{Watson1922} for $U_\nu$, it is easily
    shown~\cite{KirkpatrickNawaz2016} that
      \begin{equation}
        \EE(\sigma_i|\{\sigma_{j}\}_{j\neq i}) = f_N(|m_{(i)}|) m_{(i)}.
        \label{eq: conditional expectation of one spin}
      \end{equation}
      This then implies
    \begin{align}
        \EE[W'-W|\sigma] &= \frac{1}{n^{3/4}}mf_{N}(|m|) - \frac{1}{n}W + \frac{1}{n^{7/4}}\sum_{i=1}^{n}\left[m_{(i)}f_{N}(|m_{(i)}|) -
          mf_{N}(|m|)\right]\\
            &= -\frac{N|m|^{2}m}{n^{3/4}(N + 2)}  + \frac{1}{n^{3/4}}\left[f_{N}(|m|) - 1 + \frac{N|m|^{2}}{N+2}\right]m\\
            & \quad + \frac{1}{n^{7/4}}\sum_{i=1}^{n}\left[m_{(i)}f_{N}(|m_{(i)}|) - mf_{N}(|m|)\right],\\
            &=\frac{1}{Nn^{3/2}}\left(-\frac{N^{2}}{N+2}|W|^{2}W + R_{1}\right),
    \end{align}
    which is precisely Equation \eqref{eq: W' - W conditioned on sigma}.

    We now argue similarly to obtain \eqref{eq: W' - W W' - W T conditioned on sigma}. From the definition of $(\sigma,\sigma')$ we have
    \begin{multline}
            n^{3/2} \EE[(W'-W)(W'-W)^{\transpose}|\sigma] \\
            =\frac{1}{n}\,\sum_{i=1}^{n}\left(
            \EE[\sigma_{i}\sigma_{i}^{\transpose}|\{\sigma_j\}_{j\neq i}]
            - \EE[\sigma_{i}|\{\sigma_j\}_{j\neq i}]\sigma_{i}^{\transpose}
            - \sigma_{i}\EE[\sigma_{i}|\{\sigma_j\}_{j\neq i}]^{\transpose} + \sigma_{i}\sigma_{i}^{\transpose}
            \right).
            \label{eq: conditional expectation of W'-W projection}
    \end{multline}
    Using the same strategy as applied to prove~\eqref{eq: conditional expectation of one spin}, it can be shown~\cite{KirkpatrickNawaz2016}
    that 
    \begin{equation}
      \EE[\sigma_{i}\sigma_{i}^{\transpose}|\{\sigma_j\}_{j\neq i}] =
      \frac{1}{N}f_{N}(|m_{(i)}|) \id + g_{N}(|m_{i}|)m_{(i)}m_{(i)}^\transpose.
      \label{eq: conditional expectation of sigma-i projection}
      \end{equation}
    Substituting~\eqref{eq: conditional expectation of one spin} and~\eqref{eq: conditional expectation of sigma-i projection}
    into~\eqref{eq: conditional expectation of W'-W projection} then yields 
    \begin{multline}
      \label{eq: W' - W W' - W T calculation}
      n^{3/2}\, \EE[(W'-W)(W'-W)^{\transpose}|\sigma]\\
      = \frac{1}{n}\sum_{i=1}^{n}\left(
      f_{N}(|m_{(i)}|) \left[ \frac{\id}{N} - \left(\sigma_{i}m_{(i)}^{\transpose} + m_{(i)}\sigma_{i}^{\transpose}\right)\right]
        + g_{N}(|m_{(i)}|)m_{(i)}m_{(i)}^\transpose + \sigma_{i}\sigma_{i}^{\transpose}
        \right).
    \end{multline}
    Now observe that 
    \begin{equation}
        \begin{split}
          \frac{1}{n}\sum_{i=1}^{n}\left(\sigma_{i}m_{(i)}^{\transpose} + m_{(i)}\sigma_{i}^{\transpose}\right)
          &= \frac{1}{n^2}\sum_{i=1}^{n}\sum_{j\neq i}\left(\sigma_{i}\sigma_{j}^{\transpose} + \sigma_{j}\sigma_{i}^{\transpose}\right),\\
          &= \frac{2}{\sqrt{n}}WW^{\transpose} - \frac{2}{n^{2}}\sum_{i=1}^{n}\sigma_{i}\sigma_{i}^{\transpose}.
        \end{split}
    \end{equation}
    Using~\eqref{eq: fN + gN identity}, it then follows that
    \begin{equation}
      \begin{split}
        n^{3/2}\, \EE[(W'-W)(W'-W)^{\transpose}|\sigma]
        &= \frac{\id}{N} - \frac{2}{\sqrt{n}}WW^{\transpose} + \frac{2}{n^{2}}\sum_{i=1}^{n}\sigma_{i}\sigma_{i}^{\transpose} \\
        &\quad +\frac{1}{n}\sum_{i=1}^{n}\left(1- f_{N}(|m_{(i)}|) \right)
        \left[\sigma_{i}m_{(i)}^{\transpose} +m_{(i)}\sigma_{i}^{\transpose} - \frac{\id}{N}\right]
        \\
        &\quad + \frac{1}{n}\sum_{i=1}^n g_{N}(|m_{(i)}|)m_{(i)}m_{(i)}^\transpose + \frac{1}{n}\sum_{i=1}^n
        \sigma_{i}\sigma_{i}^{\transpose} 
      \\
      &=\frac{2}{N}(\id + R_{2})
      \end{split}
    \end{equation}
    which establishes \eqref{eq: W' - W W' - W T conditioned on sigma}.
\end{proof}

We are now in a position to prove Theorem \ref{Theorem: Critical O(N)}.

\begin{proof}[Proof of Theorem \ref{Theorem: Critical O(N)}.]
It is easily verified that $V(x)=a_N |x|^4$ satisfies Assumption~\ref{assumption: assumptions on V} with $\rho(x)=4 a_N |x|^{2}$ and
$M_{1}=M_{2}=48 a_N$. It therefore suffices to upper bound the three terms on the right-hand side of~\eqref{Theorem: Critical O(N)}, with
$W,W',R_1,R_2$ as given in Lemma~\ref{lemma: Critical remainders} and with $\lambda=N^{-1}n^{-3/2}$.

For the first term, we note that $|\delta|\leq~2n^{-3/4}$, and suppose initially that $n>2$. It then follows that $|\delta|\le 1$ and so
$|\log|\delta||=\log(1/|\delta|)$. But an elementary argument shows that $x^3[\log(1/x)\vee 1]\le 2 x^{8/3}$ for all $0\le x\le 1$. Therefore
    \begin{equation}
      \frac{1}{\lambda}\mathbb{E} [|\delta|^{3}(|\log|\delta|| \vee 1)]
      \leq N\, n^{3/2}\,\EE \,2|\delta|^{8/3}
      \leq \frac{16N}{\sqrt{n}}.
      \label{eq: third moment bound penultimate}
    \end{equation}
    But since $x^3[|\log(x)|\vee 1]\le 8$ for all  $0\le x \le 2$, it then follows from~\eqref{eq: third moment bound penultimate} that for
    all $n\in\posint$ 
    \begin{equation}
      \frac{1}{\lambda}\mathbb{E} [|\delta|^{3}(|\log|\delta|| \vee 1)]
      \leq \frac{32N}{\sqrt{n}}.
      \label{eq: third moment bound}
    \end{equation}
    
    We next consider $\EE|R_{1}|$. Since $|m|,|m_{(i)}|\in[0,1]$, and since $f_N$ is Lipschitz on $[0,1]$, there exists $c_1\in(0,\infty)$
    such that  
    \begin{equation}\label{eq: bound f N}
        |f_{N}(|m_{(i)}|) - f_{N}(|m|)| \leq c_1\left||m_{(i)}| - |m|\right| \leq c_1|m_{(i)} - m| =\frac{c_1}{n}
    \end{equation}
    where the last inequality follows from the reverse triangle inequality. It follows that there exists $c_2\in(0,\infty)$ such that
    \begin{equation}\label{eq: bound first term remainder 1}
    \begin{split}
      \frac{N}{n^{1/4}}&\left|\sum_{i=1}^{n}\left(m_{(i)}f_{N}(|m_{(i)}|) - mf_{N}(|m|)\right)\right|\\
      &= \frac{N}{n^{1/4}}\left|\sum_{i=1}^n m_{(i)}(f_N(|m_{(i)}|)-f_N(|m|)) + \sum_{i=1}^n(m_{(i)}-m)f_N(m)\right|\\
      &\leq \frac{N}{n^{1/4}}\left(|m|f_{N}(|m|)+\sum_{i=1}^{n}|m_{(i)}|\frac{c_1}{n} \right),\\ 
        &\leq \frac{N}{n^{1/4}}\left(\left[f_N(|m|)+c_1\right]\frac{|W|}{n^{1/4}} + \frac{c_1}{n} \right),\\
        &\leq c_2\frac{|W|+1}{\sqrt{n}}
    \end{split}
    \end{equation}
    where we have used the fact that $f_{N}$ is bounded on $[0,1]$.

    Taylor expanding $f_N(x)$ to second order around $x=0$ using~\eqref{eq: fN asymptotics} implies 
    there exists a $\xi\in[0,|m|]$ and $c_3\in(0,\infty)$ such that 
    \begin{equation}\label{eq: bound second term remainder 1}
        \begin{split}
            \left|Nn^{3/4}\left(f_{N}(|m|) - 1 + \frac{N|m|^{2}}{(N+2)}\right)m\right| &= Nn^{3/4}\frac{|f_{N}^{(4)}(\xi)|}{24}|m|^{5},\\
            & \leq c_3\frac{|W|^5}{\sqrt{n}},
        \end{split}
    \end{equation}
    since $f_N^{(4)}$ is bounded on $[0,1]$.

    Taking the expectation of the bounds in \eqref{eq: bound first term remainder 1} and \eqref{eq: bound second term remainder 1}
    and utilising Part~\ref{prop-part: finite moments} of Proposition~\ref{prop: mgf convergence} then implies there exists
    $r_1\in(0,\infty)$ such that
    \begin{equation}
      \label{eq: R1 bound}
        \EE\left|R_{1}\right| \leq \frac{r_1}{\sqrt{n}}.
    \end{equation}

    We now consider $\EE\|R_2\|_{\hs}$. To this end, we decompose $R_2$ into five terms as follows
    \begin{equation}
        \begin{split}
            R_{2} &= \left[ -\frac{\id}{2} + \frac{N}{2n}\sum_{i=1}^{n}\sigma_{i}\sigma_{i}^{\transpose}\right] -
            \frac{N}{\sqrt{n}}WW^{\transpose} + \frac{N}{n^{2}}\sum_{i=1}^{n}\sigma_{i}\sigma_{i}^{\transpose} ,\\ 
            & \quad + \frac{N}{2n}\sum_{i=1}^{n} \,|m_{(i)}|^2\,g_{N}(|m_{(i)}|) \,
            \left[\sigma_{i}m_{(i)}^{\transpose}+ m_{(i)} \sigma_{i}^{\transpose} - \frac{\id}{N}\right]
            + \frac{N}{2n}\sum_{i=1}^{n}g_{N}(|m_{(i)}|)m_{(i)}m_{(i)}^\transpose,\\ 
            &=: A_{1} + A_{2} + A_{3} + A_{4} + A_{5},
        \end{split}
    \end{equation}
    and bound each $A_{i}$ in turn. We defer the bound of $A_{1}$ since it is the most involved. To bound $\EE\|A_{2}\|_\hs$ and
    $\EE\|A_{3}\|_\hs$, we simply note that $\|xy^{\transpose}\|_{\hs} = |x|\,|y|$ for all $x,y \in \RR^{N}$, and so Part~\ref{prop-part:
      finite moments} of Proposition~\ref{prop: mgf convergence} implies there exist $C_2\in(0,\infty)$ such that  
    \begin{equation}
      \EE \|A_2\|_\hs = N\frac{\EE |W|^2}{\sqrt{n}} \le \frac{C_2}{\sqrt{n}},
      \label{eq: A2 bound}
    \end{equation}
    and
    \begin{equation}
      \EE \|A_3\|_\hs \le \frac{N}{n}.
      \label{eq: A3 bound}
    \end{equation}
    Similarly, 
    \begin{equation}\label{eq: A4 set up}
    \begin{split}
      \|A_{4}\|
      &\le
      \frac{N}{2n}\sum_{i=1}^n g_N(|m_{(i)}|) |m_{(i)}|^2
      \left(\|\sigma_i m_{(i)}^\transpose\| + \|m_{(i)}\sigma_i^\transpose\|+\frac{\|\id\|}{N}\right),\\
      &\le
      \frac{N}{n}\sum_{i=1}^n g_N(|m_{(i)}|) \left(\frac{|W|^2}{\sqrt{n}}+\frac{1}{n^2}\right)      
      \left(2|m_{(i)}|+\frac{1}{\sqrt{N}}\right),\\
      &\le
      N\left(\frac{|W|^2}{\sqrt{n}}+\frac{1}{n^2}\right)\left(2+\frac{1}{\sqrt{N}}\right)
      \frac{1}{n}\sum_{i=1}^n g_N(|m_{(i)}|),
    \end{split}
    \end{equation}
    and 
    \begin{equation}
      \begin{split}
        \|A_5\|_\hs &\le \frac{N}{2n}\sum_{i=1}^n g_N(|m_{(i)}|) |m_{(i)}|^2, \\
        &\le N\left(\frac{|W|^2}{\sqrt{n}}+\frac{1}{n^2}\right)\frac{1}{n}\sum_{i=1}^n g_N(|m_{(i)}|).
      \end{split}
    \end{equation}
    Since $g_N$ is bounded on $[0,1]$, Part~\ref{prop-part: finite moments} of Proposition~\ref{prop: mgf convergence} implies there
    exists $C_4, C_5\in(0,\infty)$ such that
    \begin{equation}\label{eq: A4 bound}
      \EE\|A_4\|_{\hs} \le \frac{C_4}{\sqrt{n}},
    \end{equation}
    and
    \begin{equation}\label{eq: A5 bound}
      \EE\|A_5\|_{\hs} \le \frac{C_5}{\sqrt{n}}.
    \end{equation}

    Finally, we turn our attention to $\EE\|A_{1}\|_{\hs}$. Since $\|A_{1}\|_{\hs} = \sqrt{\trace(A_{1}A_{1}^{\transpose})} =
    \sqrt{\trace(A_{1}^{2})}$, Jensen's inequality for concave functions implies
    \begin{equation}
      \begin{split}
        \label{eq initial A1 bound}
      \EE\|A_{1}\|_{\hs} &\leq
      \frac{1}{2n}\sqrt{\sum_{i,j=1}^{n}\EE\,\trace\,\left[(N\sigma_{i}\sigma_{i}^{\transpose}-\id)(N\sigma_{j}\sigma_{j}^{\transpose}-\id)\right]},\\ 
        &= \frac{1}{2n}\sqrt{\sum_{i, j = 1}^{n}\EE\,\trace\left[N^{2}\sigma_{i}\sigma_{i}^{\transpose}\sigma_{j}\sigma_{j}^{\transpose} -
            N\sigma_{i}\sigma_{i}^{\transpose} - N\sigma_{j}\sigma_{j}^{\transpose} + \id\right]},\\ 
        &= \frac{1}{2n}\sqrt{\sum_{i, j = 1}^{n}\left|N^{2}\,\EE\left[\< \sigma_{i},\sigma_{j} \>^{2}\right]-N\right|}.
    \end{split}
    \end{equation}
    If $i=j$, then we simply have $N^{2}\< \sigma_{i},\sigma_{i} \>^{2}-N=N^{2}-N$.
    Suppose then that $i \neq j$. From~\eqref{eq: conditional expectation of sigma-i projection} we have
    \begin{equation}\label{eq: Expectation of inner product squared i diff from j}
        \begin{split}
            \EE\left[\< \sigma_{i},\sigma_{j} \>^{2}\right] &= \EE[\sigma_{j}^{\transpose}\sigma_{i}\sigma_{i}^{\transpose}\sigma_{j}],\\
            &= \EE[\sigma_{j}^{\transpose}\EE[\sigma_{i}\sigma_{i}^{\transpose}|\{\sigma_{k}\}_{k \neq i}]\sigma_{j}],\\
            &= \EE\left[\sigma_{j}^{\transpose}\left(\frac{1}{N}f_{N}(|m_{(i)}|)\id +
              g_{N}(|m_{(i)}|)m_{(i)}m_{(i)}^\transpose\right)\sigma_{j}\right],\\
            &= \EE\left[\frac{1}{N}f_{N}(|m_{(i)}|) + g_{N}(|m_{(i)}|)\sigma_{j}^{\transpose}m_{(i)}m_{(i)}^\transpose\sigma_{j}\right], 
        \end{split}
    \end{equation}
    Applying~\eqref{eq: fN + gN identity} to~\eqref{eq: Expectation of inner product squared i diff from j} it follows that
    \begin{equation}\label{eq: i diff from j expression 1}
        \begin{split}
          N^{2}\EE\left[\< \sigma_{i},\sigma_{j} \>^{2}\right] - N
          &= N\,\EE\left[g_{N}(|m_{(i)}|)\left(N\<\sigma_{j}, m_{(i)} \>^{2} - |m_{(i)}|^{2}\right)\right],\\ 
         & = N\,\EE\left[g_{N}(|m_{(j)}|)\left(N\<\sigma_{j}, m_{(j)} \>^{2} - |m_{(j)}|^{2}\right)\right] + R_{i,j},
        \end{split}
    \end{equation}
    where
    \begin{equation}\label{eq: R i j}
    \begin{split}
      R_{i,j} &:= N\,\EE\left[
        g_{N}(|m_{(i)}|)\left(N\<\sigma_{j}, m_{(i)} \>^{2} - |m_{(i)}|^{2}\right)
       -g_{N}(|m_{(j)}|)\left(N\<\sigma_{j}, m_{(j)} \>^{2} - |m_{(j)}|^{2}\right)\right].
    \end{split}
    \end{equation}
    And again making use of~\eqref{eq: fN + gN identity} and~\eqref{eq: conditional expectation of sigma-i projection} shows that 
    \begin{equation}
    \begin{split}
        \EE&\left[g_{N}(|m_{(j)}|)\left(N \<\sigma_{j}, m_{(j)} \>^{2} - |m_{(j)}|^{2}\right)\right]\\
&=\EE\left[g_{N}(|m_{(j)}|)\,m_{(j)}^{\transpose} \left(N\,\EE[\sigma_{j}\sigma_{j}^{\transpose}|\{\sigma_{k}\}_{k \neq j}] - \id\right)m_{(j)}\right],\\ 
        &= \EE\left[g_{N}^2(|m_{(j)}|) m_{(j)}^{\transpose}\left(N m_{(j)} m_{(j)}^\transpose -|m_{(j)}|^2\id\right) m_{(j)}\right],\\ 
        &= (N-1)\,\EE\left[g_{N}^{2}(|m_{(j)}|)\,|m_{(j)}|^{4}\right].
    \end{split}
    \end{equation}
    Therefore, since $g_N$ is bounded on $[0,1]$ and $|m_{(j)}|^4\le 8 n^{-1} |W|^4 + 8 n^{-4}$,
    Part~\ref{prop-part: finite moments} of Proposition~\ref{prop: mgf convergence} implies there
    exists $c_4\in(0,\infty)$ such that
    \begin{equation}\label{eq: bound main term}
        N\,\EE\left[g_{N}(|m_{(j)}|)(N\<\sigma_{j}, m_{(j)} \>^{2} - |m_{(j)}|^{2})\right]
        \leq \frac{c_4}{n}.
    \end{equation}

    All that remains is to bound $R_{i, j}$ defined in \eqref{eq: R i j}. Defining $J:\RR^N\to\RR$ by
    \begin{equation}
      J(y)
      :=N\<\sigma_{j}, y \>^{2} - |y|^{2} = \sigma_{j}^{\transpose}(Nyy^{\transpose} - |y|^{2})\sigma_{j},
    \end{equation}
    it then follows that
    \begin{equation}\label{eq: rewriting R i j}
    \begin{split}
        \frac{1}{N}R_{i, j} &= \EE[g_{N}(|m_{(i)}|)J(m_{(i)}) - g_{N}(|m_{(j)}|)J(m_{(j)})],\\
        &= \EE\left[g_{N}(|m_{(i)}|)\left(J(m_{(i)}) - J(m_{(j)})\right)\right]
        + \EE\left[\left(g_{N}(|m_{(i)}|) - g_{N}(|m_{(j)}|)\right)J(m_{(j)})\right].
    \end{split}
    \end{equation}
    But the triangle inequality implies
    \begin{equation}
    \begin{split}
        |J(m_{(i)}) - J(m_{(j)})| &\le N\,\left|\frac{\< \sigma_{j}, \sigma_{i}\>^2}{n^{2}}-\frac{1}{n^{2}} + \frac{2\< \sigma_{j},
          m \>}{n} - \frac{2\< \sigma_{j}, \sigma_{i} \> \< \sigma_{j}, m \>}{n}\right|
        +\left||m_{(j)}|^{2}-|m_{(i)}|^{2}\right|,\\
        &\le N\,\left(\frac{6}{n} + \left||m_{(j)}|^{2}-|m_{(i)}|^{2}\right|\right)\\
        &\leq \frac{10N}{n},
    \end{split}
    \end{equation}
    and so since $g_N$ is bounded on $[0,1]$, there exists $c_5\in(0,\infty)$ such that
    \begin{equation}\label{eq: first R i j}
        \left|\EE\left(g_{N}(|m_{(i)}|)\left[J(m_{(i)}) - J(m_{(j)})\right]\right)\right| \leq \frac{c_5}{n}.
    \end{equation}

    Similarly, since $J(x)$ is bounded on $[-1, 1]^{N}$ and $g_N$ is Lipschitz on $[0, 1]$, there exists $c_6\in(0,\infty)$ such that
    \begin{equation}\label{eq: second R i j}
        \left|\EE\left[\left(g_{N}(|m_{(i)}|) - g_{N}(|m_{(j)}|)\right)J(m_{(j)})\right]\right| \leq \frac{c_6}{n}.
    \end{equation}
   Combining \eqref{eq: first R i j} and \eqref{eq: second R i j} with \eqref{eq: bound main term}, we conclude from~\eqref{eq initial A1
     bound} and~\eqref{eq: i diff from j expression 1} that there exists
    $C_1\in(0,\infty)$ such that
    \begin{equation}\label{eq: A1 bound}
    \EE\|A_{1}\|_{\hs} \leq \frac{C_1}{\sqrt{n}}.
    \end{equation}
    Finally, combining Equations~\eqref{eq: A1 bound}, \eqref{eq: A2 bound}, \eqref{eq: A3 bound}, \eqref{eq: A4 bound} and \eqref{eq: A5
      bound} implies there exists $r_2\in(0,\infty)$ such that 
    \begin{equation}
      \label{eq: R2 bound}
      \EE\|R_2\|_\hs \le \frac{r_2}{\sqrt{n}}.
    \end{equation}

    The stated result now follows by substituting the bounds~\eqref{eq: third moment bound}, \eqref{eq: R1 bound} and~\eqref{eq: R2 bound} into
    Theorem~\ref{Theorem: main result}.
\end{proof}

\section{SDEs and semigroups}\label{sec: sdes and semigroups}
Throughout this section and the remaining sections, we assume $V$ satisfies Assumption \ref{assumption: assumptions on V}. 
In order to study solutions of the Stein equation~\eqref{eq: Stein equation} corresponding to $V$, we follow the strategy
of~\cite{FangShaoXu2019} and consider the overdamped Langevin equation with drift $-\grad V$. We then introduce the corresponding Markov
semigroup and use the Bismut-Elworthy-Li formulae to bound derivatives of the solutions of the Stein equation~\eqref{eq: Stein equation} in terms of spatial
derivatives of the corresponding stochastic flow. 

To begin, we observe that if $V$ satisfies Assumption~\ref{assumption: assumptions on V}, then $\mu$ has finite absolute moments of all
orders. In particular, this immediately implies that $\mu(h)$ is well defined for all $h\in\lipone$. 

\begin{lemma}\label{lemma: finite moments of mu}
  Let $V$ satisfy Assumption~\ref{assumption: assumptions on V}. Then for all $p\ge 0$ we have
  $$
  \int_{\RR^{d}}|x|^{p} e^{-V(x)} \diff x < \infty.
  $$
\end{lemma}
\begin{proof}
  Combining the fundamental theorem of calculus with equations~\eqref{eq: rho bounds} and \eqref{eq: Assumption 1} we obtain
    \begin{equation}
        \begin{split}
            V(x) - V(0) & = \int_{0}^{1}\< \grad V(tx) - \grad V(0), x \> \, \dt + \< \grad V(0), x \>,\\
            & = \int_{0}^{1}\int_{0}^{1}t\<  \hess  V(stx) x, x \>\, \ds \, \dt + \< \grad V(0), x \>,\\
            & \geq c_1\,|x|^{k+2} \int_{0}^{1}\int_{0}^{1}t^{k+1}s^k\,\ind(st|x| \geq B) \, \ds \, \dt + \< \grad V(0), x \>.
        \end{split}
        \label{eq: V(x)-V(0) bound}
    \end{equation}
    Now, observe that
    \begin{equation}
      \label{eq: indicator bound in my proof}
      \ind(st|x| \geq B) \ge \ind(st|x| \geq B)\, \ind(|x|\ge 4B) \ge \ind(s\ge1/2)\ind(t\ge1/2)\, \ind(|x|\ge 4B).
    \end{equation}
    Combining equations~\eqref{eq: V(x)-V(0) bound} and~\eqref{eq: indicator bound in my proof}, it then follows that there exists $a>0$ and
    $b\in\RR$ such that
    such that
    \begin{equation}
        \begin{split}
            V(x) & \geq a\,|x|^{k+2}\ind(|x|\geq 4B) + \< \grad V(0), x \> + V(0),\\
            & \geq a\,|x|^{k+2}\ - 4^{k+2}B^{k+2}a + \< \grad V(0), x \> + V(0),\\
            & \geq a|x|^{k+2} - b.
        \end{split}
    \end{equation}
    Since $k\ge0$, we then conclude that
    \begin{equation}
    \begin{split}
        \int_{\RR^{d}}|x|^{p}\,e^{-V(x)}\diff x \leq  \int_{\RR^{d}}|x|^{p}e^{-a|x|^{k+2} + b}\, \dx < \infty.
    \end{split}
    \end{equation}
\end{proof}

We remark that the $p=0$ case of Lemma~\ref{lemma: finite moments of mu} shows that $e^{-V}$ can indeed be normalised to yield a probability measure, so that
Assumption~\ref{assumption: assumptions on V} guarantees $\mu$ is well-defined.

\subsection{Overdamped Langevin equation}
Let $B_{t}$ be a $d$-dimensional Brownian motion defined on a filtered probability space $(\Omega, \mathcal{F}, (\mathcal{F}_{t})_{t \geq
  0}, \PP)$, let $V$ satisfy Assumption~\ref{assumption: assumptions on V}, and set $g:=-\grad V$. 
The corresponding overdamped Langevin equation is the stochastic differential equation 
\begin{equation}\label{eq: overdamped langevin}
      \diff X_{t} = g(X_{t})\dt + \sqrt{2}\,  \diff B_{t}.
\end{equation}
When $X_0=x$, for fixed $x\in\RR^d$, we denote the corresponding solution to~\eqref{eq: overdamped langevin} by $X_t^x$.

Recalling the notation of~\eqref{eq: D notation}, for $u,v,w,x\in\RR^d$ we further introduce 
\begin{align}
  \frac{\diff\, \sU_u^x(t)}{\diff \,t} &= Dg(X_t^x)[\sU_u^x(t)], \qquad \sU_u^x(0)=u
  \label{eq: 1st variation}\\
  \frac{\diff\, \sU_{u,v}^x(t)}{\diff \,t} &= Dg(X_t^x)[\sU_{u,v}^x(t)] + D^2g(X_t^x)[\sU_{u}^x(t),\sU_v^x(t)], \qquad \sU_{u,v}^x(0)=0
  \label{eq: 2nd variation}\\
  \frac{\diff\, \sU_{u,v,w}^x(t)}{\diff \,t} &= Dg(X_t^x)[\sU_{u,v,w}^x(t)]
  + \frac{1}{4}\sum_{\pi\in \sym\{u,v,w\}} D^2 g(X_t^x)[\sU_{\pi(u)}^x(t),\sU_{\pi(v),\pi(w)}^x(t)]\nonumber\\
  &\quad+ D^3g(X_t^x)[\sU_{u}^x(t),\sU_v^x(t),\sU_w^x(t)], \qquad \sU_{u,v,w}^x(0)=0
  \label{eq: 3rd variation}
\end{align}
where $\sym(S)$ denotes the symmetric group of the finite set $S$.
Equations~\eqref{eq: 1st variation}, \eqref{eq: 2nd variation}, and~\eqref{eq: 3rd variation} are respectively referred to as the first,
second and third variation equations~\cite{Cerrai2001} corresponding to~\eqref{eq: overdamped langevin}.

The semigroup corresponding to~\eqref{eq: overdamped langevin} is the family of operators $P_t:B_b\to B_b$ defined for
$t\ge0$ so that for each $\varphi\in B_b$ and $x\in\RR^d$
\begin{equation}
  \label{eq: semigroup definition}
  P_t \varphi(x):=\EE\varphi(X_t^x).
\end{equation} 

The following results are consequences of general results presented in~\cite{Cerrai2001,IkedaWatanabe1981,KaratzasShreve2014}. We present a proof
in Appendix~\ref{sec: process and semigroup properties}. 
\begin{proposition}
  \label{prop: sde properties}
  Let $V$ satisfy Assumption~\ref{assumption: assumptions on V} and let $\mu$ be the probability measure with Lebesgue density proportional
  to $e^{-V}$. Then:
  \begin{enumerate}[label=(\roman*)]
  \item \label{prop_part: well posed} The equation~\eqref{eq: overdamped langevin} is well posed
  \item\label{prop_part: strong solution} For each fixed $x,u,v,w\in\RR^d$, equations~\eqref{eq: overdamped langevin}, \eqref{eq: 1st
    variation}, \eqref{eq: 2nd variation} and~\eqref{eq: 3rd variation} have unique solutions
  \item\label{prop_part: flow derivatives} The process $X_t^x$ is 3-times mean-square differentiable with respect to $x$, and for all
    $u,v,w\in\RR^d$ we have $D_u X_t^x=\sU_{u}^x(t)$, $D_v D_u X_t^x =\sU_{u,v}^x(t)$ and $D_w D_v D_u X_t^x =
    \sU_{u,v,w}^x(t)$  
  \item\label{prop_part: solution to parabolic problem} For each $h\in\smoothlip$, the unique classical solution to the corresponding parabolic problem
    \begin{equation}
      \begin{split}
        \frac{\partial u(t, x)}{\partial t} &= \mathrm{L}\, u(t, x),  \qquad t>0, \,x\in \RR^{d},\\
        u(0, x) &= h(x),  \quad \qquad x\in \RR^{d},
        \end{split}
    \end{equation}
    is given by $u(t,x) = P_t h(x)$.
  \item\label{prop_part: semigroup spatial differentiability} For each $h\in\smoothlip$, and each fixed $t\ge0$, if $u(t,x)=P_t h(x)$ then
    $u(t,\cdot)\in C^2(\RR^d)$
  \item\label{prop_part: semigroup temporal differentiability} For each $h\in\smoothlip$, and each fixed $x\in\RR^d$, if $u(t,x)=P_t h(x)$ then
    $u(\cdot,x)\in C^1(0,\infty)$
  \end{enumerate}
\end{proposition}

The strategy now is to show that for any $h\in\smoothlip$, a solution to~\eqref{eq: Stein equation} can be constructed from $P_th$, and that
this solution satisfies the properties required by Assumption~\ref{Assumption: bounds on stein solution}. To that end, we begin with the
following geometric ergodicity bound.

\begin{lemma}\label{lemma: geom erg}
  Let $V$ satisfy Assumption~\ref{assumption: assumptions on V} and let $\mu$ be the probability measure with Lebesgue density proportional
  to $e^{-V}$. The semigroup $P_{t}$ corresponding to~\eqref{eq: overdamped langevin} satisfies
    \begin{equation}\label{eq: geom erg}
        \sup_{h \in \lipone}\left|P_{t}h(x) - \mu(h)\right| \leq 2\,\left(|x| + m_{1}(\mu)\right)\,e^{-\eta t}
    \end{equation}
    where $m_{1}(\mu) < \infty$ denotes the first absolute moment of $\mu$, and
    \begin{equation}\label{eq: eta definition}
        \eta := \left(\frac{1}{4B^2}\right)\wedge\left(\frac{c_1}{4(k+1)}\right)^{2/(k+2)}.
    \end{equation}
\end{lemma}

\begin{proof}
  The proof is essentially an application of~\cite[Corollary 2]{Eberle2016}. Indeed, Example 1 of~\cite{Eberle2016} corresponds precisely to
  the process~\eqref{eq: overdamped langevin} with $V$ assumed to be $C^2$ and strictly convex outside a ball; both of these conditions are guaranteed by
  Assumption~\ref{assumption: assumptions on V}.

  The key quantity to study is $\kappa:(0,\infty)\to\RR$, defined by
    \begin{equation}
        \kappa(r):=\inf\left\{\frac{1}{|x-y|^{2}}\int_{0}^{1}\left\< x-y, \hess V(y+t(x-y))(x-y)\right\>\, \dt: x, y \in
        \RR^{d},\, |x-y| = r\right\}.
        \label{eq: Eberle kappa definition}
    \end{equation}
    Equation~\eqref{eq: Assumption 1} immediately implies $\kappa(r) \geq 0$ for all $r >0$. It then follows that
    \begin{equation}
      R_{0} := \inf\{R> 0: \kappa(r) \geq 0, \quad \forall  \ r\geq R\} = 0,
      \label{eq: R_0 value}
    \end{equation}
    and
    \begin{equation}
      \varphi(r) := \exp\left(-\frac{1}{4}\int_{0}^{r}s \kappa(s)^{-}\, \ds\right) = 1, \quad \forall \ r > 0.
      \label{eq: Eberle phi value}
    \end{equation}
    Defining
    \begin{equation}
      \label{eq: Eberle R_1 definition}
      R_{1} := \inf\{R> 0: \kappa(r)R^2 \geq 8 \quad \forall  \ r\geq R\},
    \end{equation}
    and using~\eqref{eq: R_0 value} and~\eqref{eq: Eberle phi value}, it then follows from~\cite[Corollary 2]{Eberle2016} and the
    stationarity of $\mu$ that
    \begin{equation}\label{eq: first bound on geom erg}
    \begin{split}
        \sup_{h \in \lipone}\left|P_{t}h(x) - \mu(h)\right| &\leq 2e^{-4 t/R_{1}^{2}} d(\delta_x,\mu)\\
        &\leq 2e^{-4 t/R_{1}^{2}}\left(|x| + m_{1}(\mu)\right).
    \end{split}
    \end{equation}
    In obtaining the second inequality we applied~\eqref{eq: Wasserstein definition}.

    It now remains to establish an upper bound on $R_1$. To achieve this, we begin by deriving a lower bound on $\kappa(r)$. 
    Let $e_1\in\RR^d$ denote the standard unit vector along the first coordinate axis, i.e. $e_1:=(1,0,\ldots,0)$, and let $W_1$ be
    the span of $e_1$. By \eqref{eq: Assumption 1} we have
    \begin{align}
      \frac{\kappa(r)}{c_1} &\geq
      \inf\left\{\int_{0}^{1}\ind(|y+t(x-y)| \geq B)|y+t(x-y)|^{k}\, \dt: x, y \in \RR^{d}, |x-y| = r\right\},\nonumber\\
      &=\inf\left\{\int_{0}^{1}\ind(|y+tr A e_1| \geq B)|y+trA e_1|^{k}\, \dt: y \in \RR^{d}, A\in SO(d)\right\},\nonumber\\
      &=\inf\left\{\int_{0}^{1}\ind(|y+t r e_1| \geq B)|y+t r e_1|^{k}\, \dt: y \in \RR^{d} \right\},\nonumber\\
      &=\inf\left\{\int_{0}^{1}\ind(|(a+r t)e_1+b| \geq B)|(a+rt)e_1+b|^{k}\, \dt: a\in\RR, b\in W_1^\perp \right\}\nonumber\\
      &=: \zeta(r).
    \end{align}
    where in the third step we used $|y+ rtA e_1|=|A(A^{-1}y+rte_1)|=|A^{-1}y+rte_1|$ for all $A\in SO(d)$.
    We next observe that $|(a+rt)e_1+b|^2=(a+rt)^2+|b|^2\ge(a+rt)^2$ for all $b\in W_1^\perp$. Therefore, since
    $s\mapsto s^k\indicator(s\ge B)$ is increasing on $[0,\infty)$, it follows that
    \begin{equation}
      \label{eq: zeta simplification prelim}
      \begin{split}
        \zeta(r)&=\inf_{a\in\RR} \int_{0}^{1}|a+rt|^{k}\,\ind(|a+r t| \geq B)\, \dt, \\
        &=\frac{1}{r}\inf_{a\in\RR} \int_{a-r/2}^{a+r/2}\,|s|^{k}\,\ind(|s| \geq B) \dt.
      \end{split}
    \end{equation}
    But if $r/2\le B$, then 
    \begin{equation}
      \frac{1}{r}\inf_{a\in\RR} \int_{a-r/2}^{a+r/2}\,|s|^{k}\,\ind(|s| \geq B) \dt
      \le 
       \frac{1}{r}\int_{-r/2}^{r/2}\,|s|^{k}\,\ind(|s| \geq B) \dt = 0,
    \end{equation}
    and so we conclude that
    \begin{equation}
      \label{eq: zeta simplification}
      \zeta(r)
        = \begin{cases}
          0, & r\le 2B,\\
         \displaystyle \dfrac{1}{r} \inf_{a\in\RR} \int_{U_a} |s|^{k}\, \dt , &  r>2B,
        \end{cases}
    \end{equation}
    where
    \begin{equation}
      U_a:=[a-r/2,a+r/2]\cap\{s\in\RR: |s|\ge B\}.
    \end{equation}
    Now observe that for all $a\in\RR$ we have $|U_a|\ge |U_0|$, and $|t|^k\ge|s|^k$ for all $t\in U_a\setminus U_0$ and $s\in U_0$. It
    follows that the infimum in~\eqref{eq: zeta simplification} is attained when $a=0$. Hence
    \begin{equation}
      \zeta(r) = \ind(r>2B)\frac{2r^k}{(k+1)}\left(\frac{1}{2^{k+1}}-\frac{B^{k+1}}{r^{k+1}}\right),
    \end{equation}
    and therefore
    \begin{equation}
      \kappa(r) \ge \frac{c_{1}}{2^{k+1}(k+1)}\,\ind(r\ge 4B)\,r^{k}.
      \label{eq: kappa bound}
    \end{equation}
    
    Finally, we apply the bound~\eqref{eq: kappa bound} on $\kappa(r)$ to obtain a bound on $R_1$.
    Combining~\eqref{eq: Eberle R_1 definition} and~\eqref{eq: kappa bound} yields
    \begin{equation}\label{eq: Eberle R1 bound}
        \begin{split}
          R_{1} &\le \inf\left\{R \geq 0: \ind(r\ge 4B) r^k \ge \frac{2^{k+4}(k+1)}{c_1}, \qquad \forall r\ge R\right\}\\
          &= \inf\left\{R\ge4B : R^2 r^k \ge \frac{2^{k+4}(k+1)}{c_1}, \qquad \forall r\ge R\right\} \\
          &= (4B) \vee \left(\frac{2^{k+4}(k+1)}{c_1}\right)^{1/(k+2)}. 
        \end{split}
    \end{equation}
    The stated result then follows by combining~\eqref{eq: first bound on geom erg} and~\eqref{eq: Eberle R1 bound}.
\end{proof}

\section{Proof of Theorem \ref{Theorem: main result}}\label{sec: proof of main result}
In this section, we show that for any $h\in\smoothlip$ a solution of the corresponding Stein equation~\eqref{eq: Stein equation} can be
constructed from $P_t h$. We then demonstrate that this solution satisfies the conditions of Assumption~\ref{assumption: assumptions
  on V}. In light of Theorem~\ref{Theorem: Wasserstein Bound}, this will then establish Theorem~\ref{Theorem: main result}. The
key technical challenge is to obtain the following bounds on the derivatives of $P_t h(\cdot)$.
\begin{proposition}
  \label{lemma: derivative of semigroup}
  Let $V$ satisfy Assumption~\ref{assumption: assumptions on V}, let $P_t$ denote the semigroup corresponding to~\eqref{eq: overdamped langevin},
  and let $h \in \smoothlip$. For any $x, u, v, w \in \RR^{d}$ the following hold:
    \begin{align}\label{eq: Ptphi derivative 1}
        \left|D P_{t}h(x)[u]\right| &\leq C_{1}\, e^{-\theta t}\, |u|,
        \\
        \label{eq: Ptphi derivative 2}        
         |D^{2} P_{t}h(x)[u, v]| &\leq S(t)\, e^{-\theta t/2}\, |u||v|,\\
           \label{eq: Ptphi derivative 3}
        |D^{3} P_{t}h(x)[u, v, w]| &\leq Q(t)\, e^{-\theta t/4}\, |u||v||w|, 
    \end{align}
    where $S:(0,\infty)\to(0,\infty)$ and $Q:(0,\infty)\to(0,\infty)$ are given by
    \begin{equation}
      S(t):= \frac{s_{-1/2}}{\sqrt{t}} + s_0
    \end{equation}
    and
    \begin{equation}
      Q(t):= \frac{q_{-1}}{t} + \frac{q_{-1/2}}{\sqrt{t}} + q_0 + q_1 t + q_2 t^2
    \end{equation}
    where $\theta,C_1,s_{-1/2},s_0,q_{-1},q_{-1/2},q_0,q_1,q_2\in(0,\infty)$. 
\end{proposition}
The proof of Proposition~\ref{lemma: derivative of semigroup} is presented in Section~\ref{sec: semigroup bounds}. It 
uses Bismut-Elworthy-Li formulae to express the derivatives of $P_th$ in terms of the processes defined by the variation
equations~\ref{Assumption: bounds on stein solution}, together with estimates of the latter established in Section~\ref{sec: variation
  process bounds}. 

\begin{lemma}\label{lemma: expressions for the derivatives of f}
  Let $V$ satisfy Assumption~\ref{assumption: assumptions on V} and let $P_t$ denote the semigroup corresponding to~\eqref{eq: overdamped langevin}.  
  Fix $h \in\smoothlip$, and define
  \begin{equation}\label{eq: solution to stein eq}
    f(x):=-\int_{0}^{\infty}[P_{t}h(x) - \mu(h)]\, \dt.
  \end{equation}
  Then $f \in C^2(\RR^{d})$, and for all $x,u,v\in\RR^d$ 
    \begin{equation}\label{eq: first derivative of f}
      D_{u}f(x) = -\int_{0}^{\infty}D P_{t}h(x)[u]\, \dt,
    \end{equation}
    and 
    \begin{equation}\label{eq: second derivative of f}
        D_{v} D_{u}f(x) = -\int_{0}^{\infty}D^{2} P_{t}h(x)[u, v] \, \dt.
    \end{equation}
    Moreover, $f$ is a solution to the corresponding Stein equation~\eqref{eq: Stein equation}.
\end{lemma}
\begin{proof}
  Part~\ref{prop_part: semigroup spatial differentiability} of Proposition~\ref{prop: sde properties} implies $P_t h(\cdot)-\mu(h)\in
  C^2(\RR^d)$, and~\eqref{eq: Ptphi derivative 1} implies
    \begin{equation}
      |D_{u}[P_{t}h(x) - \mu(h)]| = |D P_{t}h(x)[u]| \leq C_{1}e^{-\theta t }|u|. 
    \end{equation}
    Since the upper bound is integrable, equation~\eqref{eq: first derivative of f} then follows by dominated convergence.
    Similarly, since $-D_{u}P_t h(\cdot)\in C^1(\RR^d)$, and the right-hand side of equation~\eqref{eq: Ptphi derivative 2} is
    integrable, dominated convergence establishes equation~\eqref{eq: second derivative of f}.

    Furthermore, the continuity of $-D_{v}D_{u} P_t h(\cdot)$ and the integrability of the right-hand side of equation~\eqref{eq: Ptphi
      derivative 2} imply, again by dominated convergence, that the right-hand side of~\eqref{eq: second derivative of f} is continuous. In
    particular, it then follows from~\eqref{eq: second derivative of f} that all second-order partial derivatives of $f$ are continuous, and
    so $f\in C^2(\RR^d)$. 
   
    Finally, a consequence of equations~\eqref{eq: first derivative of f} and~\eqref{eq: second derivative of f} is that
    \begin{equation}
     \ig f(x) = -\int_0^\infty \ig P_t h(x) \dt,
    \end{equation}
    and it then follows from Part~\ref{prop_part: solution to parabolic problem} of Lemma~\ref{prop: sde properties} that in fact
    \begin{equation}
      \ig f(x) = -\int_0^\infty \frac{\partial}{\partial t} P_t h(x) \dt.
      \label{eq: stein solution after FTC}
    \end{equation}
    But using Part~\ref{prop_part: semigroup temporal differentiability} of Proposition~\ref{prop: sde properties}, it follows from the
    fundamental theorem of calculus that for any $T>0$ we have
    \begin{equation}
      \label{eq: finite T FTC}
      -\int_0^T \frac{\partial}{\partial t} P_t h(x) \dt = h(x) - P_T h(x).
    \end{equation}
    Taking the limit $T\to\infty$ using Lemma~\ref{lemma: geom erg} we then obtain
    \begin{equation}
      \label{eq: limit of FTC}
      -\int_0^\infty \frac{\partial}{\partial t} P_t h(x) \dt = h(x) - \mu(h).
    \end{equation}
    Combining equations~\eqref{eq: stein solution after FTC} and~\eqref{eq: limit of FTC} then shows that $f$ indeed satisfies the Stein
    equation~\eqref{eq: Stein equation}. 
\end{proof}

Armed with Proposition~\ref{lemma: derivative of semigroup} and Lemma~\ref{lemma: expressions for the derivatives of f}, we are now in a
position to prove Theorem \ref{Theorem: main result}. 

\begin{proof}[Proof of Theorem \ref{Theorem: main result}]
  Thanks to Theorem \ref{Theorem: Wasserstein Bound} and Lemma~\ref{lemma: finite moments of mu}, it suffices to show that there exists
  $C\in(0,\infty)$ such that for every $h\in\smoothlip$, the corresponding Stein solution~\eqref{eq: solution to stein eq} satisfies the
  bounds~\eqref{eq: bound gradient f}, \eqref{eq: Bound hilbert schmidt norm} and~\eqref{eq: Bound hilbert schmidt norm 2} for all
  $u,v\in\RR^d$ and $\epsilon>0$.

  Let $h\in\smoothlip$, and $x, u, v \in \RR^{d}$.
  Combining~\eqref{eq: first derivative of f} and~\eqref{eq: Ptphi derivative 1} immediately yields
    \begin{equation}
      |D_{u}f(x)| \leq \int_{0}^{\infty}|D P_t h(x)[u]|\, \dt \leq \frac{C_{1}}{\theta}|u|,
      \label{eq: 1st grad f bound}
    \end{equation}
    so that~\eqref{eq: bound gradient f} holds with $C \geq K_1:=C_{1}/\theta$. Similarly, combining~\eqref{eq: second derivative of f} and
    \eqref{eq: Ptphi derivative 2} yields
    \begin{equation}\label{eq: 2nd grad f bound}
        \begin{split}
          |D_{v}D_{u}f(x)| &\leq \int_{0}^{\infty}\left| D^{2}P_{t}h(x)[u, v]\right|\, \dt,\\
          &\leq\int_{0}^{\infty} S(t)\, e^{-\theta t/2}\,\dt \, |u||v|,\\
          &= K_2\, |u||v|,
        \end{split}
    \end{equation}
    with
    \begin{equation}
      K_2:=\left(\frac{2 s_0}{\theta }+\sqrt{\frac{2\pi}{\theta }} s_{-1/2}\right),
    \end{equation}
    which shows that~\eqref{eq: Bound hilbert schmidt norm} holds with $C\ge K_2$. 
    
    Finally, to establish~\eqref{eq: Bound hilbert schmidt norm 2}, we fix $\epsilon>0$ and $y\in\RR^d$ with $|y| \leq 1$ and apply a
    similar decomposition to that used in~\cite{FangShaoXu2019}. Specifically, defining
    \begin{equation}\label{eq: Phi definition}
            \Phi(t) :=\left[D^{2}P_{t}h(x + \epsilon y)[u, v] - D^{2}P_{t}h(x)[u, v]\right],
    \end{equation}
    it follows by~\eqref{eq: second derivative of f} that
    \begin{equation}\label{eq: rewriting f(x + epsilon u) - f(x)}
        D_{v}D_{u}f(x + \epsilon y) - D_{v}D_{u}f(x) = -\int_{0}^{\epsilon^{2}\wedge \epsilon}\Phi(t)\,\dt -
        \int_{\epsilon^{2}\wedge \epsilon}^{\infty}\Phi(t)\,\dt.
    \end{equation}
    Equation~\eqref{eq: Ptphi derivative 2} implies
    \begin{equation}\label{eq: bound int 0 to epsilon squared}
    \begin{split}
      \left|\int_{0}^{\epsilon^{2}\wedge \epsilon}\Phi(t) \, \dt\right| &\leq 2|u||v|\int_{0}^{\epsilon^{2}\wedge \epsilon} S(t)\, \dt,\\
      &\leq 2\,|u| |v|\,(\epsilon\wedge \sqrt{\epsilon})\,  \left(s_0( \epsilon\wedge \sqrt{\epsilon}) +2s_{-1/2}\right).
    \end{split}
    \end{equation}
    If $\epsilon\le 1$ then $\epsilon\wedge \sqrt{\epsilon}=\epsilon$, in which case
    \begin{equation}
      (\epsilon\wedge \sqrt{\epsilon})\,  \left(s_0( \epsilon\wedge \sqrt{\epsilon}) +2s_{-1/2}\right) = 
      \epsilon\,  (s_0\epsilon +2s_{-1/2})\,\leq \epsilon (s_0 + 2s_{-1/2}).
    \end{equation}
    While if $\epsilon>1$ then $\epsilon\wedge\sqrt{\epsilon}=\sqrt{\epsilon}$, and so 
    \begin{equation}
      (\epsilon\wedge \sqrt{\epsilon})\,  \left(s_0( \epsilon\wedge \sqrt{\epsilon}) +2s_{-1/2}\right) =
      \sqrt{\epsilon}\,  \left(s_0\sqrt{\epsilon}+2s_{-1/2}\right)\leq
      \epsilon\left(s_{0} +2s_{-1/2}\right).
    \end{equation}
    We therefore conclude from~\eqref{eq: bound int 0 to epsilon squared} that for all $\epsilon>0$
    \begin{equation}\label{eq: int bound 0 to min}
        \left|\int_{0}^{\epsilon^{2}\wedge \epsilon}\Phi(t) \, \dt\right| \leq 2\epsilon (s_0 + 2s_{-1/2})|u| |v|.
    \end{equation}
    Now consider the remaining term in~\eqref{eq: rewriting f(x + epsilon u) - f(x)}. It follows from~\eqref{eq: Ptphi derivative 3} that
    $D^2P_t h(x+\cdot)[u,v]$ is differentiable with bounded derivative, and so the fundamental theorem of calculus implies 
    \begin{equation}
      \int_{\epsilon^{2}\wedge \epsilon}^{\infty}\Phi(t)\,\dt
      = \epsilon\int_{\epsilon^{2}\wedge \epsilon}^{\infty}\int_{0}^{1} D^{3}P_{t}h(x + \epsilon r y)[u, v, y]\, \dr\, \dt. 
    \end{equation}
    Therefore, using $|y|\le1$, and again applying~\eqref{eq: Ptphi derivative 3}, yields
    \begin{equation}
      \label{eq: phi tail bound}
    \begin{split}
      \left|\int_{\epsilon^{2}\wedge \epsilon}^{\infty}\Phi\,\dt\right| 
      &\leq \epsilon |u||v| \int_{\epsilon\wedge\epsilon^{2}}^{\infty}\, Q(t) \,e^{-\theta t/4} \dt,\\
      &\leq \epsilon |u||v|
      \left[\frac{128q_2}{\theta^3}+\frac{16q_1}{\theta^2}+\frac{4q_0}{\theta}+\frac{2\sqrt{\pi}q_{-1/2}}{\sqrt{\theta}}
      + q_{-1} E_1\left(\frac{\theta(\epsilon^{2}\wedge \epsilon)}{4}\right)\right]
    \end{split}
    \end{equation}
    where $E_{1}(\cdot)$ denotes the exponential integral~\cite[Equation 5.1.1]{AbramowitzStegun1972}. Applying~\cite[Equation 5.1.20]{AbramowitzStegun1972} 
    yields
    \begin{equation}
      \label{eq: bound on exponential integral}
      \begin{split}
        E_1\left(\frac{\theta(\epsilon^{2}\wedge \epsilon)}{4}\right) &\le 2\left(1\vee \left|\log \frac{\theta(\epsilon^{2}\wedge \epsilon)}{4} \right|\right)\\
        &\le 4\left(1+\left|\log\frac{\theta}{4}\right|\right)\left(1\vee |\log\epsilon|\right).
      \end{split}
    \end{equation}
    Applying~\eqref{eq: bound on exponential integral} to~\eqref{eq: phi tail bound}, and combining the result with~\eqref{eq: int bound 0
      to min} and~\eqref{eq: rewriting f(x + epsilon u) - f(x)} then yields 
    \begin{equation}
      \sup_{x,y\in\RR^d : |y|\le 1} |D_{v}D_{u}f(x + \epsilon y) - D_{v}D_{u}f(x)|
      \le
      K_3\,|u|\, |v|\, \epsilon\, (1\vee|\log\epsilon|)
      \label{eq: third main bound}
    \end{equation}
    with
    \begin{equation}
      K_3 :=
      2(s_0 +2s_{-1/2}) + 
      \frac{128q_2}{\theta^3}+\frac{16q_1}{\theta^2}+\frac{4q_0}{\theta}+\frac{2\sqrt{\pi}q_{-1/2}}{\sqrt{\theta}}
      + 4\left(1+\left|\log\frac{\theta}{4}\right|\right) q_{-1}.
    \end{equation}

    Combining Equations~\eqref{eq: 1st grad f bound}, \eqref{eq: 2nd grad f bound} and~\eqref{eq: third main bound} then implies the stated result holds with
    $C:=\max\{K_1,K_2,K_3\}$.
\end{proof}

\section{Bounding the solutions of the variation equations}\label{sec: variation process bounds}
The key ingredient needed to prove Proposition~\ref{lemma: derivative of semigroup}
is the following result, which gives bounds on the solutions to the variation equations~\eqref{eq: 1st variation}, \eqref{eq: 2nd variation}
and~\eqref{eq: 3rd variation}. 
\begin{proposition}\label{prop: mean-square bounds}
If $V$ satisfies Assumption~\ref{assumption: assumptions on V}, then for $x, u, v, w\in \RR^{d}$ and $t>0$,
\begin{equation}\label{eq: grad u X_t deterministic}
     |\,\sU_{u}^x(t)| \leq |u|,
 \end{equation}
\begin{equation}\label{eq: grad u X_t}
     \EE|\,\sU_{u}^x(t)|^{2} \leq C_{1}^{2}e^{-2\theta t}|u|^{2},
 \end{equation}
\begin{equation}\label{eq: grad u1 u2 X_t}
     \EE|\,\sU_{u,v}^x(t)|^{2} \leq C_{2}^{2}|u|^{2}|v|^{2},
 \end{equation}
 \begin{equation}\label{eq: grad u1 u2 u3 X_t deterministic}
     |\,\sU_{u,v,w}^x(t)| \leq |u||v||w|P(t),
 \end{equation}
 where, 
 \begin{align*}
     C_{1} &:= \frac{1}{\sqrt{J(2t_{0})}},                               &\theta :=& \frac{\log(1/J(2t_{0}))}{2t_{1}},\\ 
     C_{2} &:= \sqrt{2}M_1\left(1+\frac{C_1^2}{\theta^2}\right)^{1/2},   &P(t)   :=& (3M_1^2(t+1)/2+M_2)(t+1).
 \end{align*}
with $J:[1,\infty)\to\RR$ defined via
   \begin{equation*}
     J(t):= 1-\frac{1-e^{-\chi}}{2t}, \qquad  \chi:=c_1B^k,
   \end{equation*}
   and 
   \begin{equation*}
     \begin{split}
        t_{0}&:= \frac{\eta + 2B + 2m_{1}(\mu)}{\eta\mu(\{x\in\RR^{d}: |x| > B+ 1\})},\\
        t_{1}&:= \max\left\{2t_{0}, \frac{1}{\chi}\log\left(\frac{4t_{0}}{1-e^{-\chi}}\right)\right\}.
        \end{split}
    \end{equation*}
   The constants $k,B,c_1,M_1,M_2$ are as defined in Assumption~\ref{assumption: assumptions on V}, while $\mu$ and $\eta$ are as defined in
   Lemma~\ref{lemma: geom erg}. 
\end{proposition}

The analogue of Proposition~\ref{prop: mean-square bounds} in~\cite{FangShaoXu2019} is Lemma 5.2. While the proof of the latter is
essentially immediate under their assumption of strict convexity of $V$,
the proof of Proposition~\ref{prop: mean-square bounds} under Assumption~\ref{assumption: assumptions on V} is somewhat more involved. A key
ingredient in its proof is the following lemma, whose proof is deferred to the end of this section. 
\begin{lemma}\label{lemma: E(t) bound}
  If $V$ satisfies Assumption~\ref{assumption: assumptions on V}, then for all $x\in\RR^d$ and $0\le s\le t<\infty$ 
  $$
  \EE \exp\left(-2\int_s^t \rho(X_r^x) \dr \right) \le C_2^2 e^{-2\theta (t-s)}.
  $$
\end{lemma}
In addition, we will also require the following Gr\"{o}nwall-type lemma.
\begin{lemma}\label{lemma: Gronwall type}
    Let $u: \RR \to \RR^{d}$ be such that $|u(t)|^{2} \in C^{1}(\RR, \RR)$. Suppose it satisfies the following differential inequality
    \begin{equation}
        \frac{\diff}{\dt}|u(t)|^{2} \leq 2a(t)|u(t)|^{2} + 2b(t)|u(t)|, \qquad \forall t \in (0, T),
    \end{equation}
    where $T \in (0, \infty)$, $a, b \in L^{1}([0, T])$ and $b \geq 0$. If $|u(0)| = 0$ then
    \begin{equation}
        |u(t)| \leq \int_{0}^{t}b(s)\exp\left(\int_{s}^{t}a(r)\, \dr\right)\, \ds, \qquad \forall t \in [0, T).
    \end{equation}
\end{lemma}
\begin{proof}
     Fix $t \in (0, T)$ and let $\tau := \sup\{0 \leq s \leq t: |u(s)| = 0\}$. If $\tau = t$ then the statement holds trivially since $b
     \geq 0$. If $\tau < t$ then $|u(t)|$ is differentiable on $(\tau, t]$ and gives 
    \begin{equation}
        \frac{\diff}{\ds}|u(s)| \leq a(s)|u(s)| + b(s), \qquad s \in (\tau, t].
    \end{equation}
    Since $|u(\tau)| = 0$, by the Generalised Jones Inequality \cite[Theorem 1.2.2]{Qin2016} we have
    \begin{equation}
         |u(t)| \leq \int_{\tau}^{t}b(s)\exp\left(\int_{s}^{t}a(r)\, \dr\right)\, \ds \leq  \int_{0}^{t}b(s)\exp\left(\int_{s}^{t}a(r)\, \dr\right)\, \ds,
    \end{equation}
    using $b \geq 0$. This completes the proof.
\end{proof}

Armed with Lemmas~\ref{lemma: E(t) bound} and~\ref{lemma: Gronwall type}, we are now ready to prove Proposition~\ref{prop: mean-square bounds}.
\begin{proof}[Proof of Proposition~\ref{prop: mean-square bounds}]
  We begin by noting that combining~\eqref{eq: 1st variation} and~\eqref{eq: Assumption 1} yields
\begin{equation}\label{eq: differential equations}
    \frac{\diff}{\dt}|\,\sU_u^x(t)|^{2} \le -2\,\rho(X_{t}^{x})\,|\,\sU_u^x(t)|^2,
\end{equation}
and so it follows from Gr\"{o}nwall's lemma that
\begin{equation}\label{eq: Upper bound grad u X}
    |\,\sU_u^x(t)|^{2} \leq \exp\left(-2\int_{0}^{t}\rho(X_{s}^{x})\, \ds\right)|u|^{2}, \qquad t\ge0.
\end{equation} 
Since $\rho$ is nonnegative, we immediately obtain the first stated result, \eqref{eq: grad u X_t deterministic}.
Moreover, by taking the expectation and applying Lemma~\ref{lemma: E(t) bound}, equation~\eqref{eq: Upper bound grad u X} also yields
\begin{equation}\label{eq: bound 4th moment grad u X}
  \begin{split}
    \EE\,|\,\sU_u^x(t)|^{2} &\leq E(t) |u|^2,\\
    &\le C_1^2\,e^{-2\theta t}\,|u|^2,
    \end{split}
\end{equation}
which establishes the second stated result, \eqref{eq: grad u X_t}. 

We now consider the second variation equation. It follows from~\eqref{eq: 2nd variation} and the Cauchy-Schwarz inequality that
\begin{equation}
  \begin{split}
    \frac{1}{2}\frac{\diff}{\dt} |\,\sU_{u,v}^x(t)|^2 &=
    \left\<\sU_{u,v}^x(t), Dg(X_t^x)[\sU_{u,v}^x(t)]\right\> + \left\<\,\sU_{u,v}^x(t),D^2g(X_t^x)[\,\sU_u^x(t),\,\sU_v^x(t)]\right\>\\
    &\le
    \left\<\sU_{u,v}^x(t), Dg(X_t^x) \sU_{u,v}^x(t) \right\> + |\,\sU_{u,v}^x(t)|\,|D^2g(X_t^x)[\,\sU_u^x(t),\,\sU_v^x(t)]|
  \end{split}
\end{equation}
Applying~\eqref{eq: Assumption 1} and~\eqref{eq: Assumption 2} then yields
\begin{equation}
  \begin{split}
    \frac{1}{2}\frac{\diff}{\dt} |\,\sU_{u,v}^x(t)|^2
    &\le -\rho(X_t^x) |\,\sU_{u,v}^x(t)|^2 + M_1(1+\rho(X_t^x)\,|\,\sU_u^x(t)|\,|\,\sU_v^x(t)|\,|\,\sU_{u,v}^x(t)|,\\
    &\le -\rho(X_t^x) |\,\sU_{u,v}^x(t)|^2 + M_1\,|u|\,|v|\,(1+\rho(X_t^x)\,|\,\sU_{u,v}^x(t)|,\\
  \end{split}
\end{equation}
where in the last step we made use of~\eqref{eq: grad u X_t deterministic}. It now follows from Lemma~\ref{lemma: Gronwall type} that
    \begin{equation}\label{eq: inequality u(t)}
    \begin{split}
        |\,\sU^{x}_{u,v}(t)| &\leq M_{1}|u||v|\int_{0}^{t}(1 + \rho(X_{s}^{x}))\exp\left(-\int_{s}^{t}\rho(X_{r}^{x})\, \dr\right)\, \ds,\\
    \end{split}
    \end{equation}
    But since $\rho(X_s^x)$ is a continuous function of $s$, the fundamental theorem of calculus implies
    \begin{equation}
      \label{eq: rho ftc step}
      \int_{0}^{t}\rho(X_{s}^{x})\exp\left(-\int_{s}^{t}\rho(X_{r}^{x})\, \dr\right)\, \ds = 1 -
    \exp\left(-\int_{0}^{t}\rho(X_{r}^{x})\, \dr\right) <  1.
    \end{equation}
    Therefore, setting
    \begin{equation}
      Q_t^x:=\int_{0}^{t}\exp\left(-\int_{s}^{t}\rho(X_{r}^{x})\right)\,\ds,
    \end{equation}
    we obtain
    \begin{equation}
      \label{eq: second variation in terms of Q}
      |\,\sU^{x}_{u,v}(t)| \leq M_{1}|u||v|\,(1+Q_t^x).
    \end{equation}
    Now, the Cauchy-Schwarz inequality and Lemma~\ref{lemma: E(t) bound} imply that
    \begin{equation}\label{eq: inequality f(t) squared}
        \begin{split}
            \EE (Q_t^x)^{2} &= \EE \int_{0}^{t}\int_{0}^{t}\exp\left(-\int_{s_{1}}^{t}\rho(X_{r}^{x})\,
            \dr\right)\exp\left(-\int_{s_{2}}^{t}\rho(X_{r}^{x})\, \dr\right)\, \diff s_{1} \diff s_{2},\\ 
            & \leq \int_{0}^{t}\int_{0}^{t}\left[\EE\exp\left(-2\int_{s_{1}}^{t}\rho(X_{r}^{x})\, \dr\right) \EE
              \exp\left(-2\int_{s_{2}}^{t}\rho(X_{r}^{x})\, \dr\right)\right]^{1/2}\, \diff s_{1} \diff s_{2},\\ 
            &\leq C_{1}^{2}\int_{0}^{t}\int_{0}^{t}e^{-\theta(t-s_{1})}e^{-\theta(t-s_{2})}\diff s_{1}\diff s_{2},\\
            &= C_{1}^{2}\left(\int_{0}^{t}e^{-\theta(t-s)}\, \ds\right)^{2}, \\
            &\leq \frac{C_{1}^{2}}{\theta^{2}}.
        \end{split}
    \end{equation}
    Combining \eqref{eq: second variation in terms of Q} and \eqref{eq: inequality f(t) squared} then yields
    \begin{equation}
    \begin{split}
        \EE|\,\sU^{x}_{u,v}(t)|^{2} &\le M_{1}^{2}|u|^{2}|v|^{2}\EE(1 + Q_t^x)^{2},\\
        &\leq 2M_{1}^{2}|u|^{2}|v|^{2}(1 + \EE (Q_t^x)^{2}),\\
        &\leq 2M_{1}^{2}\left(1 + \frac{C_{1}^{2}}{\theta^{2}}\right)|u|^{2}|v|^{2},
    \end{split}
    \end{equation}
    which establishes the third stated result, \eqref{eq: grad u1 u2 X_t}.

    It remains only to bound the solution of the third variation equation. It follows from~\eqref{eq: 3rd variation} and the Cauchy-Schwarz
    inequality, followed by application of equations~\eqref{eq: Assumption 1}, \eqref{eq: Assumption 2} and \eqref{eq: Assumption 3}, that
    \begin{equation}
      \begin{split}
        \frac{1}{2}\frac{\diff}{\dt}|\,\sU_{u,v,w}^x(t)|^2
        &= \left\<\,\sU_{u,v,w}^x(t), Dg(X_t^x)\,\sU_{u,v,w}^x(t)\right\> \\
      &\quad +\frac{1}{4}\sum_{\pi\in\sym\{u,v,w\}}\left\<\,\sU_{u,v,w}^x(t), D^2
      g(X_t^x)[\,\sU_{\pi(u)}^x(t),\sU_{\pi(v),\pi(w)}^x(t)]\right\> \\
      &\quad + \left\<\,\sU_{u,v,w}^x(t),D^3g(X_t^x)[\,\sU_u^x(t),\,\sU_v^x(t),\,\sU_w^x(t)]\right\>,\\
      &\le \left\<\,\sU_{u,v,w}^x(t), Dg(X_t^x)\,\sU_{u,v,w}^x(t)\right\> \\
      &\quad +\frac{1}{4}\sum_{\pi\in\sym\{u,v,w\}}|\,\sU_{u,v,w}^x(t)|\, |D^2 g(X_t^x)[\,\sU_{\pi(u)}^x(t),\sU_{\pi(v),\pi(w)}^x(t)]| \\
      &\quad + |\,\sU_{u,v,w}^x(t)|\,|D^3g(X_t^x)[\,\sU_u^x(t),\,\sU_v^x(t),\,\sU_w^x(t)]|,\\
      &\le -\rho(X_t^x)\,|\,\sU_{u,v,w}^x(t)|^2\\
      &\quad +\frac{M_1}{4}\left(1+\rho(X_t^x)\right) \,|\,\sU_{u,v,w}^x(t)|\,\sum_{\pi\in\sym\{u,v,w\}}|\,\sU_{\pi(u)}^x(t)|\,|\,\sU_{\pi(v),\pi(w)}^x(t)| \\
      &\quad + M_2|\,\sU_{u,v,w}^x(t)|\,\left(1+\rho(X_t^x)\right) |\,\sU_u^x(t)|\,|\,\sU_v^x(t)|\,|\,\sU_w^x(t)|.\\
      \end{split}
    \end{equation}
    Now utilising equation~\eqref{eq: grad u X_t deterministic}, and equation~\eqref{eq: second variation in terms of Q} together with the
    fact that $Q_t^x\le t$, results in
    \begin{equation}
      \begin{split}
        \frac{1}{2}\frac{\diff}{\dt}|\,\sU_{u,v,w}^x(t)|^2
        &\le
        -\rho(X_t^x) |\,\sU_{u,v,w}^x(t)|^2 \\
        &\quad + \left[M_2+\frac{3}{2}M_1^2(1+t)\right]\left(1+\rho(X_t^x)\right)\,|u|\,|v|\,|w|\,|\,\sU_{u,v,w}^x(t)|.
        \end{split}
    \end{equation}
    From Lemma \ref{lemma: Gronwall type} and equation~\eqref{eq: rho ftc step} we then conclude that
    \begin{equation}
    \begin{split}
      |\sU^{x}_{u,v,w}(t)| & \leq |u||v||w|\int_{0}^{t}
      \left[M_2+\frac{3}{2}M_1^2(1+s)\right]\left(1+\rho(X_s^x)\right) \exp\left(-\int_{s}^{t}\rho(X_{r}^{x})\, \dr\right)\, \ds,\\
      &\le
      |u||v||w|\left[M_2+\frac{3}{2}M_1^2(1+t)\right]
      \int_{0}^{t} \left(1+\rho(X_s^x)\right) \exp\left(-\int_{s}^{t}\rho(X_{r}^{x})\, \dr\right)\, \ds,\\
      &\le
      |u||v||w|\left[M_2+\frac{3}{2}M_1^2(1+t)\right](1+t),      
    \end{split}
    \end{equation}
    which establishes the final stated result, \eqref{eq: grad u1 u2 u3 X_t deterministic}.
\end{proof}

\begin{proof}[Proof of Lemma~\ref{lemma: E(t) bound}]
  It is convenient in this proof to work on the canonical space $(C[0,\infty)^d,\sB(C[0,\infty)^d))$, where
  $C[0,\infty)^d$ is the space of all continuous maps $\omega:[0,\infty)\to\RR^d$, and $\sB(C[0,\infty)^d)$ is the $\sigma$-algebra
  generated by finite-dimensional cylinder sets. Moreover, for each $t\ge0$ let $X_t:C[0,\infty)^d\to\RR^d$ denote the coordinate map
  $X_t(\omega)=\omega(t)$, and let $\sB_t:=\sigma(X_s: s\le t)$. 

  Since, by Part~\ref{prop_part: well posed} of Proposition~\ref{prop: sde properties}, equation~\eqref{eq: overdamped langevin} is
  well posed, it induces a family of measures $(\PP_x)_{x\in\RR^d}$ on $(C[0,\infty)^d,\sB(C[0,\infty)^d))$ and we then have, for each
      $x\in\RR^d$ and $t\ge0$ that 
      \begin{equation}
        \label{eq: unconscious statistician}
      \EE\left[\exp\left(-2\int_{s}^{t}\rho(X_{r}^x)\, \dr\right)\right]
      =
      \EE_x\left[\exp\left(-2\int_{s}^{t}\rho(X_{r})\, \dr\right)\right].
    \end{equation}
    where the expectation $\EE_x$ on the right-hand side is with respect to the measure $\PP_x$.
    
  Now, for $x\in\RR^d$ and $t\ge0$, define
  \begin{equation}
  E(x, t) := \EE_x\left[\exp\left(-2\int_{0}^{t}\rho(X_{s})\, \ds\right)\right],
  \end{equation}
  and
  \begin{equation}
    E(t) := \sup_{x\in\RR^d} \EE_x\left[\exp\left(-2\int_{0}^{t}\rho(X_{s})\, \ds\right)\right]
  \end{equation}
  For any $0\le s \le t <\infty$, it then follows from the Markov property and time-homogeneity that
  \begin{equation}
    \label{eq: reducing stated result to E(t) bound}
      \begin{split}
      \EE_x \exp\left(-2\int_s^t \rho(X_r) \dr \right) &= \EE_x \exp\left(-2\int_0^{t-s} \rho(X_{r+s}) \dr \right)\\
      &=\EE_x \EE_x \left[\exp\left(-2\int_0^{t-s} \rho(X_{r+s}) \dr \right)\Bigg| \sB_s \right] \\
      &=\EE_x \EE_{X_s} \exp\left(-2\int_0^{t-s} \rho(X_{r}) \dr \right) \\
      &\le E(t-s).
      \end{split}
    \end{equation}
  It therefore suffices to show that $E(t)$ decays exponentially.
  We begin by first establishing submultiplicativity of $E(t)$. Applying again the Markov property we find
    \begin{equation}\label{eq: bound on E(x,t+s)}
        \begin{split}
            E(x, t+s) &= \EE_x\left[\exp\left(-2\int_{0}^{t + s}\rho(X_{r})\, \dr\right)\right],\\
            &= \EE_x\left[\exp\left(-2\int_{0}^{t}\rho(X_{r})\, \dr\right)
              \EE_x\left[\exp\left(-2\int_{t}^{t + s}\rho(X_{r})\, \dr\right)\bigg|\sB_{t}\right]\right],\\
            &= \EE_x\left[\exp\left(-2\int_{0}^{t}\rho(X_{r})\, \dr\right)
              \EE_{X_t}\left[\exp\left(-2\int_{0}^{ s}\rho(X_{r})\, \dr\right)\right]\right],\\ 
            &\leq E(x,t)E(s).
        \end{split}
    \end{equation}
    Taking the supremum over $x\in\RR^d$ then establishes submultiplicativity
    \begin{equation}\label{eq: E t is sub multiplicative}
      E(t + s) \leq E(t)E(s).
      \end{equation}
    Moreover, since $E(s)\le1$ for all $s\ge0$, equation~\eqref{eq: bound on E(x,t+s)} also implies that for all $s, t \geq 0$
    \begin{equation}\label{eq: E x t is decreasing}
        E(x, t+s) \leq E(x, t).
    \end{equation}
    Our strategy now is to bound $E(t_1)$, with $t_1$ as given in Proposition~\ref{prop: mean-square bounds}. Combined with submultiplicativity~\eqref{eq: E t is
      sub multiplicative}, this will then imply $E(t)$ decays exponentially.

    To this end, we define the event $\Lambda_{t} := \{\int_{0}^{t}\ind(|X_{s}| \geq B)\, \ds\ > 1/2\}$, and 
    establish a lower bound on $\PP_x(\Lambda_{t})$. Let $\psi:\RR^d\to\RR$ be defined as follows
    \begin{equation}
        \psi(x):= \begin{cases}
            0, & |x| \leq B,\\
            |x| - B, & B<|x|\leq B+1,\\ 
            1, &  \text{otherwise}.
        \end{cases}
    \end{equation}
    Then $\psi \in \lipone$, so that setting $K(a):=2\left(a + m_{1}(\mu)\right)$
    Lemma \ref{lemma: geom erg} implies
    \begin{equation}
            \EE_x[\psi(X_{t})] \geq \mu(\psi) - K(|x|)e^{-\eta t}.
    \end{equation}
    And since $\psi(x) \leq \ind(|x| \geq B)$, we then have
    \begin{equation}
        \begin{split}
            \EE_x\left[\int_{0}^{t}\ind(|X_{s}| \geq B)\, \ds\right] &\geq \EE_x\left[\int_{0}^{t}\psi(X_{s})\, \ds\right],\\
            &\geq t\,\mu(\psi) - \frac{1}{\eta}K(|x|),\\
            &\geq t\,\mu(\{x\in\RR^{d}: |x| > B+ 1\}) - \frac{1}{\eta}K(|x|).
        \end{split}
    \end{equation}
    In particular, taking $t=t_0$, with $t_0$ as given in Proposition~\ref{prop: mean-square bounds},
    it follows that for all $x\in\RR^d$ with $|x|\le B$ we have
    \begin{equation}
        \begin{split}
          1 &\leq \EE_x\left[\int_{0}^{t_{0}}\ind(|X_{s}| \geq B)\, \ds\right]\\
          &=\EE_x\left[\int_{0}^{t_{0}}\ind(|X_{s}| \geq B)\, \ds\bigg|\Lambda_{t_{0}}\right]\PP_x(\Lambda_{t_{0}}) +
            \EE_x\left[\int_{0}^{t_{0}}\ind(|X_{s}| \geq B)\, \ds\bigg|\Lambda_{t_{0}}^{c}\right]\PP_x(\Lambda_{t_{0}}^{c}),\\ 
            &\leq t_{0}\,\PP_x(\Lambda_{t_{0}}) + \frac{1}{2}.
        \end{split}
    \end{equation}
    Rearranging, we then obtain
    \begin{equation}\label{eq: lower bound on P(Lambda)}
      \PP_x(\Lambda_{t_{0}}) \geq \frac{1}{2t_{0}}, \qquad x\in\RR^d, \quad|x|\le B.
    \end{equation}

    Combining~\eqref{eq: lower bound on P(Lambda)} and~\eqref{eq: rho bounds}, it then follows that for all $x\in\RR^d$ with $|x|\le B$
    \begin{equation}
        \begin{split}
            E(x, t_{0}) & = \EE_x\left[\exp\left(-2\int_{0}^{t_{0}}\rho(X_{s})\, \ds\right)\ind(\Lambda_{t_{0}})\right] +
            \EE_x\left[\exp\left(-2\int_{0}^{t_{0}}\rho(X_{s})\, \ds\right)\ind(\Lambda_{t_{0}}^{c})\right],\\ 
            &\leq e^{-\chi}\,\PP_x(\Lambda_{t_{0}}) + \PP_x(\Lambda_{t_{0}}^{c}),\\
            &= 1 - (1-e^{-\chi})\,\PP_x(\Lambda_{t_{0}}) \\
            &\leq 1-\frac{1-e^{-\chi}}{2t_{0}}\\
            &= J(t_0),
        \end{split}
    \end{equation}
    and therefore
    \begin{equation}\label{eq: small x E bound}
    \sup_{|x| \leq B}E(x, t_{0}) \leq J(t_{0}).
    \end{equation}

    Next, we use the bound~\eqref{eq: small x E bound} to derive a bound on $E(x,t)$ that holds for all $x\in\RR^d$ and sufficiently large
    $t$. To achieve this, we define the stopping time
    \begin{equation}
      \tau := \inf\{s\geq 0:|X_{s}| \leq B\}. 
    \end{equation}
    By then partitioning $E(x,t)$ into expectations on the events $\{\tau > t/2\}$ and  $\{\tau \leq t/2\}$
    we obtain
    \begin{equation}\label{Eq: first bound of E}
    \begin{split}
        E(x, t) &= \EE_x\left[\exp\left(-2\int_{0}^{t}\rho(X_{s})\, \ds\right)\ind(\tau > t/2)\right] \\
        & \quad + \EE_x\left[\exp\left(-2\int_{0}^{t}\rho(X_{s})\, \ds\right)\ind(\tau \leq t/2)\right],\\
        &\leq e^{-\chi t} + \EE_x\left[\exp\left(-2\int_{\tau}^{t}\rho(X_{s})\, \ds\right)\ind(\tau \leq t/2)\right],\\
        &\leq e^{-\chi t} + \EE_x\left[\exp\left(-2\int_{0}^{t-\tau}\rho(X_{s+\tau})\, \ds\right)\ind(\tau \leq t/2)\right],\\
        &\leq e^{-\chi t} + \EE_x\left[\exp\left(-2\int_{0}^{t/2}\rho(X_{s+\tau})\, \ds\right)\ind(\tau \leq t/2)\right],\\
        &= e^{-\chi t} + \EE_x\, \ind(\tau \leq t/2)\,\EE_x\left[\exp\left(-2\int_{0}^{t/2}\rho(X_{s+\tau})\, \ds\right)\Bigg|\sB_\tau\right].
    \end{split}
    \end{equation}
    Now, since equation~\eqref{eq: overdamped langevin} is well posed and $g=-\grad V$ is locally compact, the family
    $\{\PP_x\}_{x\in\RR^d}$ satisfies the strong Markov property~\cite[Theorem 5.4.20]{KaratzasShreve2014}. 
    It therefore follows that
    \begin{equation}\label{Eq: second bound of E}
    \begin{split}
      E(x,t) &\le e^{-\chi t} + \EE_x\, \ind(\tau \leq t/2)\,E(X_\tau,t/2)\\
      &\le e^{-\chi t} + \EE_x\, \ind(\tau \leq t/2)\,\sup_{|y|\le B} E(y,t/2)\\
      &\le e^{-\chi t} + \sup_{|y|\le B} E(y,t/2)
    \end{split}
    \end{equation}
    where in the second step we observed that continuity implies $|X_\tau|\le B$.
    We conclude from~\eqref{Eq: second bound of E} that for all $t\ge0$
    \begin{equation}
      \label{eq: E(t) bootstrap bound}
      E(t) \le e^{-\chi t} + \sup_{|y|\le B} E(y,t/2).
    \end{equation}
    Combining~\eqref{eq: E(t) bootstrap bound} with~\eqref{eq: E x t is decreasing} and~\eqref{eq: small x E bound} it follows that for all $t\ge 2t_0$
    \begin{equation}
      E(t) \le e^{-\chi t} + J(t_0).
    \end{equation}
    In particular, from the definition of $t_1$, it then immediately follows that
    \begin{equation}
      \label{eq: E t1 bound}
      E(t_1) \le J(2t_0). 
    \end{equation}

    Finally, for any $t\ge0$ set $n:=\lfloor t/t_1\rfloor$.
    Since $J(t)<1$, it then follows from~\eqref{eq: E t is sub multiplicative} and~\eqref{eq: E t1 bound} that
    \begin{equation}
      \label{eq: exponential decay of E(t)}
      \begin{split}
        E(t) &\le E(t_1)^n \\
        &\le J(2t_0)^n \\
        &\le \frac{1}{J(2t_0)}\, \left(J(2t_0)^{1/t_1}\right)^t\\
        &= C_2^2 e^{-2\theta t}.
      \end{split}
    \end{equation}
    Combining equations~\eqref{eq: unconscious statistician}, \eqref{eq: reducing stated result to E(t) bound} and~\eqref{eq: exponential
      decay of E(t)} yields the stated result.
\end{proof}

\section{Proving Proposition~\ref{lemma: derivative of semigroup}}\label{sec: semigroup bounds}
The key ingredients needed to prove Proposition~\ref{prop: mean-square bounds} are the following.
\begin{proposition}[\cite{Cerrai2001}]
  \label{prop: BEL} Let $V$ satisfy Assumption~\ref{assumption: assumptions on V}. For $u\in\RR^d$ and $t>0$, define
\begin{equation}\label{eq: Iu definition}
        I_{u}^{x}(t) := \frac{1}{\sqrt{2}t}\int_{0}^{t}\<\,\sU_{u}^x(s), \dbs \>.
\end{equation}
Then for any $\varphi \in C_{b}$, any $t>0$ and any $u,v,w\in\RR^d$ it holds that,
\begin{equation}\label{eq: semigroup chain rule}
  D_{u}P_{t}\varphi(x) = \EE\left\< \grad\varphi(X_{t}^{x}),\,\sU_{u}^x(t)\right\>,
\end{equation}
\begin{equation}\label{eq: BEL formula 2}
  \begin{split}
  D^{2} P_{t}\varphi(x)[u,v] &= \EE\left(D P_{t/2}\varphi(X_{t/2}^{x})[\sU_{u}^x(t/2)]I_{v}^{x}(t/2)\right) \\
  &\quad +\frac{2}{t}\EE\int_{0}^{t/2}D P_{t-s}\varphi(X_{s}^{x})[\sU_{u,v}^x(s)]\, \ds,
  \end{split}
\end{equation}
and 
\begin{equation}\label{eq: BEL formula 3}
  D^{3} P_{t}\varphi(x)[u,v,w] = K_1^{x,t}(u,v,w) + K_2^{x,t}(u,v,w) + K_3^{x,t}(u,v,w) + K_4^{x,t}(u,v,w) + K_4^{x,t}(v,u,w)
\end{equation}
where
\begin{align}
K_1^{x,t}(u,v,w)&:= \EE \,\left(D^{2} P_{t/2}\varphi(X_{t/2}^{x})[\,\sU_{u}^x(t/2),\sU_{v}^x(t/2)] I_{w}^{x}(t/2)\right) \\
K_2^{x,t}(u,v,w)&:= \EE \,\left(D P_{t/2}\varphi(X_{t/2}^{x})[\sU_{u,v}^x(t/2)]I_{w}^{x}(t/2)\right)\\ 
K_3^{x,t}(u,v,w)&:= \frac{2}{t}\EE\int_{0}^{t/2}D P_{t-s}\varphi(X_{s}^{x})[\,\sU_{u,v,w}^x(s)]\, \ds,\\
K_4^{x,t}(u,v,w)&:= \frac{2}{t}\EE\int_{0}^{t/2}D^{2} P_{t-s}\varphi(X_{s}^{x})[\sU_{u,v}^x(s), \,\sU_w^x(s)]\, \ds.
  \end{align}
\end{proposition}
\begin{proof}
  The chain rule, \eqref{eq: semigroup chain rule}, is classical; see e.g.~\cite[Section 3]{Cerrai2019} for a discussion.
  As shown in Lemma~\ref{lemma: Cerrai's hypotheses}, Assumption~\ref{assumption: assumptions on V}
    implies that~\cite[Hypotheses 1.1, 1.2 and 1.3]{Cerrai2001} are satisfied. It then follows from
    \cite[Section 1.5]{Cerrai2001} that the Bismut-Elworthy-Li formulae~\eqref{eq: BEL formula 2} and \eqref{eq: BEL formula 3} hold.
\end{proof}

We now prove the main result of this section.
\begin{proof}[Proof of Proposition~\ref{lemma: derivative of semigroup}]
  Fix $h \in \smoothlip$. It follows from~\eqref{eq: semigroup chain rule} that
    \begin{equation}
      \left|D_{u} P_{t}h(x)\right| = |\EE\<\grad h(X_{t}^{x}), \,\sU_u^x(t) \>| \le \EE|\,\sU_u^x(t)|.
    \end{equation}
    Applying equation~\eqref{eq: grad u X_t} of Proposition~\ref{prop: mean-square bounds} and Jensen's inequality then immediately
    yields~\eqref{eq: Ptphi derivative 1}.
    
    Now applying~\eqref{eq: Ptphi derivative 1} to \eqref{eq: BEL formula 2}, we obtain
    \begin{equation}
        \begin{split}
          |D^{2} P_{t}h(x)[u,v]|
          & \leq \EE\left(\left|D P_{t/2}h(X_{t/2}^{x})[\,\sU_u^x(t/2)]\right|\left|I_{v}^{x}(t/2)\right|\right)\\
          &\quad + \frac{2}{t}\EE\int_{0}^{t/2}\left|D P_{t-s}h(X_{s}^{x})[\,\sU_{u,v}^x(s)]\right|\, \ds,\\
          & \leq C_{1} \, e^{-\theta t/2}\,\EE\left(|\,\sU_u^x(t/2)||I_{v}^{x}(t/2)|\right)
          + \frac{2C_{1}}{t}\int_{0}^{t/2}\, e^{-\theta (t-s)}\,\EE|\,\sU_{u,v}^x(s)|\, \ds.
        \end{split}
        \label{eq: Hessian Pt first step}
    \end{equation}
    But the Burkholder-Davis-Gundy inequality~\cite[Chapter III, Theorem 3.1]{IkedaWatanabe1981}, together with equation~\eqref{eq: grad u X_t
      deterministic} of Proposition~\ref{prop: mean-square bounds} implies
    \begin{equation}\label{eq: I squared bound}
        \EE|I_{v}^{x}(t)|^2 \leq \frac{2}{t^2}\EE \int_{0}^{t}|\,\sU_v^x(s)|^{2}\, \ds \leq \frac{2}{t}|v|^2.
    \end{equation}
    Hence, applying~\eqref{eq: grad u X_t deterministic} and Jensen's inequality to the first term in~\eqref{eq: Hessian Pt first step}, and
    making use of~\eqref{eq: I squared bound} as well as \eqref{eq: grad u1 u2 X_t} of Proposition~\ref{prop: mean-square bounds}, we obtain 
    \begin{equation}
        \begin{split}
          |D^{2} P_{t}h(x)[u, v]|
          &\leq C_{1} \, |u|\,e^{-\theta t/2}\,\sqrt{\EE|I_{v}^{x}(t/2)|^{2}}
          + 2C_{1}C_{2}\frac{(1 - e^{-\theta t/2})}{\theta t}|u||v|\,e^{-\theta t/2},\\
            & \leq \frac{2C_{1}}{\sqrt{t}}\,\, e^{-\theta t/2}\,\,|u||v| + C_{1}C_{2}\, \, e^{-\theta t/2}\,\,|u||v|,
        \end{split}
    \end{equation}
    which establishes equation~\eqref{eq: Ptphi derivative 2} with $s_{-1/2}=2C_1$ and $s_0=C_1 C_2$.

    Finally we turn attention to \eqref{eq: Ptphi derivative 3}. Applying~\eqref{eq: Ptphi derivative 2}, followed
    by~\eqref{eq: grad u X_t deterministic} and~\eqref{eq: I squared bound} to $K_1^{x,t}$ we obtain 
    \begin{equation}\label{eq: final K1 bound}
      \begin{split}
        |K_1^{x,t}[u,v,w]| &\le \EE|D^2 P_{t/2} h(X_{t/2}^x)[\sU_u^x(t/2),\sU_v^x(t/2)]|\,|I_w^x(t/2)|,\\
        &\le S(t/2) e^{-\theta t/4}\,\EE |\,\sU_u^x(t/2)|\,|\,\sU_v^x(t/2)|\,|I_w^x(t/2)|, \\
        &\le S(t/2) e^{-\theta t/4}\, |u|\, |v|\,\EE |I_w^x(t/2)|, \\
        &\le \frac{2S(t/2)}{\sqrt{t}} e^{-\theta t/4}\, |u|\, |v|\,|w|.
      \end{split}
    \end{equation}
    While applying~\eqref{eq: Ptphi derivative 1}, followed by the Cauchy-Schwarz inequality and equations~\eqref{eq: I squared bound}
    and~\eqref{eq: grad u1 u2 X_t} to $K_2^{x,t}$ yields
    \begin{equation}\label{eq: final K2 bound}
        \begin{split}
          |K_{2}^{x,t}[u,v,w]|
          &\leq \EE\left|D P_{t/2}h(X_{t/2}^{x})[\,\sU_{u,v}^x(t/2)]\right| \,\left|I_{w}^{x}(t/2)\right|,\\
          &\leq C_{1} \, e^{-\theta t/2}\,\EE|\,\sU_{u,v}^x(t/2)|\,|I_{w}^{x}(t/2)|,\\
          &\leq C_{1} \, e^{-\theta t/2}\,\sqrt{\EE|\,\sU_{u,v}^x(t/2)|^2\,\EE|I_{w}^{x}(t/2)|^2},\\
          &\leq \frac{2C_1C_2}{\sqrt{t}} \, e^{-\theta t/2}\,|u|\, |v|\,|w|.
        \end{split}
    \end{equation}
    And applying~\eqref{eq: Ptphi derivative 1} and \eqref{eq: grad u1 u2 u3 X_t deterministic} to $K_3^{x,t}$ implies
    \begin{equation}\label{eq: final K3 bound}
      \begin{split}
        |K_{3}^{x,t}[u,v,w]|
        & \leq \frac{2}{t}\EE\int_{0}^{t/2}\left|D P_{t-s}h(X_{s}^{x})[\,\sU_{u,v,w}^x(s)]\right|\, \ds\\
        & \leq \frac{2}{t}\,C_1\,\int_{0}^{t/2} e^{-\theta (t-s)}\, \EE\left|\,\sU_{u,v,w}^x(s)\right|\, \ds\\         
        & \leq  C_{1}P(t/2)\left(\frac{2}{t}\int_{0}^{t/2}e^{-\theta(t-s)}\, \ds\right)|u|\,|v|\,|w|,\\
        & \leq  C_{1}P(t/2)\, e^{-\theta t/2}\, |u|\,|v|\,|w|.
      \end{split}
    \end{equation}
    where we have used the fact that $P(t)$ is increasing on $[0, \infty)$.
    While applying~\eqref{eq: Ptphi derivative 2} followed by~\eqref{eq: grad u X_t deterministic} and~\eqref{eq: grad u1 u2 X_t} yields
    \begin{equation}\label{eq: final K4 bound}
        \begin{split}
          |K_{4}^{x,t}[u,v,w]|
          &\leq \frac{2}{t}\EE\int_{0}^{t/2}\left|D^{2}P_{t-s}h(X_{s}^{x})[\,\sU_{u,v}^x(s),\,\sU_{w}^x(s)]\right|\, \ds,\\
          &\leq \frac{2}{t} \int_{0}^{t/2}S(t-s) e^{-\theta(t-s)/2}\, \EE\,|\,\sU_{u,v}^x(s)|\,|\,\sU_{w}^x(s) |\, \ds,\\
          &\leq C_2\,|u|\,|v|\,|w|\,\frac{2}{t} \int_{0}^{t/2}S(t-s) e^{-\theta(t-s)/2}\, \ds,\\
          &\leq C_2 S(t/2)\,|u|\,|v|\,|w|\,\frac{2}{t} \int_{0}^{t/2} e^{-\theta(t-s)/2}\, \ds,\\          
          &\leq C_2 S(t/2)\,|u|\,|v|\,|w|\,e^{-\theta t/4}.
        \end{split}
    \end{equation}
    Combining the estimates~\eqref{eq: final K1 bound}, \eqref{eq: final K2 bound}, \eqref{eq: final K3 bound} and~\eqref{eq: final K4
      bound} we obtain~\eqref{eq: Ptphi derivative 3} with $q_{-1}=4\sqrt{2}C_1$, $q_{-1/2}=2(2+\sqrt{2})C_1C_2$,
    $q_0=C_1(C_2^2+3M_1^2/2+M_2)$, $q_1=C_1(3M_1^2+M_2)/2$ and $q_2=3C_1M_1^2/8$.
\end{proof}

\backmatter
\bmhead{Acknowledgements}
  The authors thank Zihua Guo for useful conversations.
  This research was supported by and the Australian Research Council's Discovery Projects funding scheme (Project No. DP230102209).

\begin{appendices}

\section{Proof Proposition~\ref{prop: sde properties}}\label{sec: process and semigroup properties}
  \begin{lemma}\label{lemma: Cerrai's hypotheses}
    If $V$ satisfies Assumption~\ref{assumption: assumptions on V}, then the overdamped Langevin equation~\eqref{eq: overdamped langevin}
    satisfies Hypotheses 1.1, 1.2 and 1.3 of~\cite{Cerrai2001}.
  \end{lemma}
  Before proving Lemma~\ref{lemma: Cerrai's hypotheses}, we first apply it to prove Proposition~\ref{prop: sde properties}.
  \begin{proof}[Proof of Proposition~\ref{prop: sde properties}]
    In light of Lemma~\ref{lemma: Cerrai's hypotheses}, Part~\ref{prop_part: strong solution} follows from \cite[Theorem
      1.2.5]{Cerrai2001} and~\cite[Propositions 1.3.2, 1.3.3 and 1.3.4]{Cerrai2001} and Part~\ref{prop_part: flow derivatives} follows
    from~\cite[Proposition 1.3.5]{Cerrai2001}. Moreover parts~\ref{prop_part: solution to parabolic problem}, \ref{prop_part: semigroup
      spatial differentiability} and~\ref{prop_part: semigroup temporal differentiability} follow from \cite[Theorem 1.6.2]{Cerrai2001}  

   Finally, we consider Part~\ref{prop_part: well posed}. Since $V\in C^4(\RR^d;\RR)$, it immediately follows that $g=-\grad V$ is locally
   Lipschitz, and so~\cite[Theorem IV.3.1]{IkedaWatanabe1981} implies that~\eqref{eq: overdamped langevin} satisfies pathwise uniqueness.
   By a corollary of~\cite[Theorem
     IV.1.1]{IkedaWatanabe1981}, this then also implies uniqueness in the sense of probability law.
   Combined with Part~\ref{prop_part: strong solution}, which establishes the existence of a solution to equation~\eqref{eq: overdamped
     langevin} for each initial position $x$, this then implies that~\eqref{eq: overdamped langevin} is well posed.
  \end{proof}

  We now verify that the relevant hypotheses from~\cite{Cerrai2001} hold under Assumption~\ref{assumption: assumptions on V}.
  
  \begin{proof}[Proof of Lemma~\ref{lemma: Cerrai's hypotheses}]
    First observe that since the diffusion matrix in~\eqref{eq: overdamped langevin} is constant, it trivially follows that Hypothesis 1.3
    holds, as well as Part 2 of~\cite[Hypothesis 1.1]{Cerrai2001}.
    
    Now observe that for all $x\in\RR^d$ we have
    \begin{equation}
      1+\rho(x) \le (1+c_2)(1+|x|)^k,
    \end{equation}
    and consequently~\eqref{eq: Assumption 2} implies
    \begin{align}
        |D_{v}D_{u} g(x)| &\le M_1(1+c_2)(1+|x|)^{k-1}\,|u|\,|v|       \label{eq: Hyp1.1.1 j=2 a}\\
        &\le M_1(1+c_2)(1\vee 2^{k-2})(1+|x|^{k-1})\,|u|\,|v|,
      \label{eq: Hyp1.1.1 j=2 b}
    \end{align}
    and assumption~\eqref{eq: Assumption 3} implies
    \begin{equation}
      |D_{w}D_{v}D_{u} g(x)| \le M_2(1+c_2)(1\vee 2^{k-3})(1+|x|^{k-2})\,|u|\,|v|\,|w|.
      \label{eq: Hyp1.1.1 j=3}
    \end{equation}
    Suppose $k>0$. The fundamental theorem of calculus implies 
    \begin{equation}
      D_{u} g(x) = D_{u} g(0) + \int_0^1 D_x D_{u} g(tx)\, \diff t.
      \label{eq: ftc}
    \end{equation}
    Applying~\eqref{eq: Hyp1.1.1 j=2 a} to~\eqref{eq: ftc}, it then follows that
    \begin{equation}
      \begin{split}
        |D_{u} g(x)| &\le |D_{u} g(0)| + M_1(1+c_2)|u|\,\frac{(1+|x|)^k}{k}\\
        &\le |D_{u} g(0)| + M_1(1+c_2)\frac{(1\vee 2^{k-1})}{k} |u|\,(1+|x|^k),
      \end{split}
      \label{eq: grad u1 g bound}
    \end{equation}
    But
    \begin{equation}
      |D_{u}g(0)|\le\|\hess V(0)\|_{\mathrm{op}}\, |u|,
    \end{equation}
    and combining with~\eqref{eq: grad u1 g bound} then shows that whenever $k>0$, there exists $C\in(0,\infty)$ such that for all $x\in\RR^d$
    \begin{equation}
      |D_{u} g(x)| \le C(1+|x|^{k})\,|u|.
      \label{eq: Hyp1.1.1 j=1}
    \end{equation}
   Equation~\eqref{eq: Assumption 0} immediately implies~\eqref{eq: Hyp1.1.1 j=1} also holds in the case $k=0$, and so it in fact holds for any $k\ge0$.

   By a similar argument, combining the fundamental theorem of calculus with~\eqref{eq: Hyp1.1.1 j=1} yields
   \begin{equation}
     \begin{split}
     |g(x)| &\le |g(0)| +\int_0^1 C |x| (1+|tx|^k)\,\diff t \\
     &= (|g(0)|+ 2C)\,(1+|x|^{k+1}).
     \label{eq: Hyp1.1.1 j=0}
     \end{split}
   \end{equation}
   Combining~\eqref{eq: Hyp1.1.1 j=0}, ~\eqref{eq: Hyp1.1.1 j=1}, ~\eqref{eq: Hyp1.1.1 j=2 b} and ~\eqref{eq: Hyp1.1.1 j=3} shows that Part 1
   of~\cite[Hypothesis 1.1]{Cerrai2001} is satisfied.
   And since~\eqref{eq: Assumption 1} immediately implies $\<D g(x)[y],y\>\le 0$ for all $x,y\in\RR^d$, it follows that Part 3
   of~\cite[Hypothesis 1.1]{Cerrai2001} holds, and therefore Assumption~\ref{assumption: assumptions on V} indeed implies \cite[Hypothesis
     1.1]{Cerrai2001}. 

   Now fix $x\in\RR^d$. The fundamental theorem of calculus implies that for any $u\in\RR^d$ 
   \begin{equation}
     g(u+x)-g(x) = \int_0^1 D_u g(tu+x) \,\diff t
   \end{equation}
   and so by~\eqref{eq: Assumption 1}
   \begin{equation}
     \begin{split}
       \<g(u+x)-g(x),u\> &\le -c_1 \, |u|^2\,\int_0^1 |x+tu|^k \indicator(|x+tu|\ge B)\,\diff t.
       \label{eq: g(u+x)-g(x) bound}
       \end{split}
   \end{equation}
   But we now show that~\eqref{eq: g(u+x)-g(x) bound} in fact implies 
   \begin{equation}
     \<g(x+u)-g(x),u\>\le -\frac{c_1}{2^{2k+1}}\,|u|^{k+2}+ c_1 2^{k+4}(1+B^{k+2})(1+|x|^{k+2}).
     \label{eq: hypothesis 1.2}
   \end{equation}
   This will show that~\cite[Hypothesis 1.2]{Cerrai2001} holds.
   To show \eqref{eq: hypothesis 1.2} we consider two cases. First suppose that $|u|\ge 4(|x|+B)$. Then for all $t\in[1/2,1]$ we have $|x+tu|\ge B$, and hence~\eqref{eq:
     g(u+x)-g(x) bound} yields
   \begin{equation}
     \<g(x+u)-g(x),u\>\le -c_1\,|u|^2\,\int_{1/2}^1\left(\frac{|u|}{4}\right)^k\,\diff t = -\frac{c_1}{2^{2k+1}}\,|u|^{k+2}.
   \end{equation}
   Therefore~\eqref{eq: hypothesis 1.2} holds whenever $|u|\ge 4(|x|+B)$. Consider instead then $|u|< 4(|x|+B)$. By Jensen's inequality we
   have
   \begin{equation}
     |u|^{k+2} \le 2^{3k+5}(1+B^{k+2})(1+|x|^{k+2})
   \end{equation}
   and so
   \begin{equation}
     -\frac{1}{2^{2k+1}}|u|^{k+2} + 2^{k+4}(1+B^{k+2})(1+|x|^{k+2}) \ge 0.
   \end{equation}
   Since the right-hand side of~\eqref{eq: g(u+x)-g(x) bound} is non-positive, it then follows that
   \begin{equation}
     \<g(x+u)-g(x),u\>\le 0 \le -c_1\frac{1}{2^{2k+1}}|u|^{k+2} + c_1\,2^{k+4}(1+B^{k+2})(1+|x|^{k+2}).
   \end{equation}
   This establishes \eqref{eq: hypothesis 1.2}. We conclude that Assumption~\ref{assumption: assumptions on V} indeed implies that
   Hypotheses 1.1, 1.2, 1.3 of~\cite{Cerrai2001} all hold.
  \end{proof}
    
\end{appendices}


\end{document}